\documentclass[11pt, a4paper]{amsart}

\usepackage[colorinlistoftodos]{todonotes}

\usepackage[utf8]{inputenc}
\usepackage{lmodern}
\usepackage{amssymb,amsfonts,amsthm,amsmath,bbm,mathrsfs,aliascnt,mathtools}
\usepackage[english]{babel}
\usepackage[colorlinks]{hyperref}

\usepackage{ulem}
\usepackage{nicefrac}

\usepackage{tikz} 
\usepackage{subcaption}

\usepackage{graphicx} %
\usepackage{sidecap, caption}
\usepackage{wrapfig}

\usepackage{multicol}
\usepackage{lipsum}
\usepackage{array}

\theoremstyle{plain}
\newtheorem{thm}{Theorem}[section]
\newtheorem{cor}[thm]{Corollary}
\newtheorem{lem}[thm]{Lemma}
\newtheorem{prop}[thm]{Proposition}
\newtheorem{rem}[thm]{Remark}

\theoremstyle{definition}
\newtheorem{definition}[thm]{Definition}

\newtheorem{ex}[thm]{Example}

\theoremstyle{remark}
\newtheorem{remark}[thm]{Remark}

\DeclareMathOperator{\LIM}{LIM}

\DeclareMathOperator{\supp}{supp}

\DeclareMathOperator{\almost}{--a.e.}

\newcommand{\R}{\mathbb R}

\newcommand{\N}{\mathbb N}

\def\d{{\rm d}}

\begin{document}

\title[Proximality, Stability...]{Proximality, stability, and central limit theorem for random maps on an interval}

\author[S.C. Hille]{Sander C. Hille}
\address{Mathematical Institute, Leiden University, P.O. Box 9512, 2300 RA Leiden, The Netherlands (SH)}
\email{shille@math.leidenuniv.nl}

\author[K. Horbacz]{Katarzyna Horbacz}
\address{Institute of Mathematics,
University of Silesia, Bankowa 14,
40-007 Katowice, Poland (KH)}
\email{katarzyna.horbacz@us.edu.pl}
\author[H. Oppelmayer]{Hanna Oppelmayer}
\address{Institut f\"ur Mathematik und Geometrie, Universit\"at Innsbruck, Technikerstrasse 13, A--6020 Insbruck, Austria (HO)}
\email{Hanna.Oppelmayer@uibk.ac.at}
\author[T. Szarek]{Tomasz Szarek}
\address{Institut of Mathematics Polish Academy of Sciences, Abrahama 18, Sopot, Poland (TS)}
\email{tszarek@impan.pl}

\thanks{T.S. was supported by the Polish NCN Grant 2022/45/B/ST1/00135.}
\thanks{H.O. was partially supported by Early Stage Funding-Vicerector for research, University of Innsbruck.}

\date\today

\maketitle

{\centering\footnotesize {\it To the memory of Professor Tadeusz Figiel (1948--2025)} \par}

\begin{abstract}  Stochastic dynamical systems consisting of non-invertible continuous maps on an interval are studied. It is proved that if they satisfy the recently introduced so-called $\mu$--injectivity and some mild assumptions, then proximality in the sense of Furstenberg, asymptotic stability of the corresponding Markov operator, synchronization, and a central limit theorem hold. 
 \end{abstract}
 

\section{Introduction}
\renewcommand{\thethm}{\Alph{thm}}

The behaviour of random maps has been intensively studied recently (see \cite{Alseda-Misiurewicz, A_B_B, BBS, DKN, HKRVZ, Malicet, Navas}).
In the majority of works, however, it was assumed that the maps belong to some group of invertible transformations. The invertibility of maps $g_i$ seemed to be crucial since in proving ergodicity results this allows one to apply probabilistic tools, some martingale convergence theorems for instance, jointly with dynamical systems tools (ergodic theorems). To be precise, the martingale convergence theorem holds for right random walks $R_n=g_1\circ\cdots\circ g_n$ but the Birkhoff ergodic theorem is formulated for C\'esaro's means of left random walks $L_n=g_n\circ\cdots\circ g_1$.
Due to invertibility we may replace the left random walk with the right random walk taking inverse transformations and
using the trivial identity: $g_1\circ\cdots\circ g_n=(g_n^{-1}\circ\cdots\circ g_1^{-1})^{-1}$. Thus we can use probabilistic and dynamical systems techniques simultaneously (see for instance \cite{BBS, DKN}). This breaks down for non-invertible maps.

In this paper, the considered maps on the interval $[0, 1]$ do not have to be invertible. We will assume that they are continuous and piecewise monotone (with finitely many pieces) and are chosen randomly according to some distribution $\mu$. 
We denote the collection of these maps by $\Gamma$ and assume that $\supp\mu=\Gamma$.
We study $\mu$--invariant (also called $\mu$--stationary) Borel probability measures $\nu$ on $[0,1]$, i.e.,
$$\nu=\int_{\Gamma} \nu\circ g^{-1} \, d\mu(g).$$

Ergodicity is a key concept in the theory of dynamical systems and stochastic processes since it captures some nice statistical properties of studied systems (see \cite{L_M, O_V}). In the present paper we investigate the quantity of ergodic $\mu$--invariant Borel probability measures. 
In the work of S. Brofferio and the last two named authors \cite{BOS} we introduced the so-called \textit{$\mu$--injectivity} condition, which can be thought of as  injectivity ``on average'', namely
$$\int_{\Gamma} \#g^{-1}(\{x\})\, d\mu(g)\leq 1  \quad\text{for all } x\in[0,1],
$$ where $\# g^{-1}(\{x\})$ denotes the number of elements in the pre-image.
This condition enabled us in \cite{BOS} to adapt the techniques developed in the theory of partially hyperbolic diffeomorphism by  
A. Avila and M. Viana in \cite{A_V} and subsequently extended by D. Malicet \cite{Malicet} to random homeomorphisms on a circle and an interval. 
In \cite{BOS} we provided a sufficient condition for the existence of a unique ergodic measure within the class of $\mu$--injective systems. Furthermore, in a continuation of the present paper by the authors without Horbacz, it is provided another sufficient condition for unique ergodicity in \cite{HOS} using bounded variation. In the present paper, the technique we use to deduce unique ergodicity, is proximality in the sense of Furstenberg  (see Definition~\ref{def: prox}), which might be of independent interest.
 Proximal dynamical systems gained special attention due to the monumental works of Furstenberg \cite{F63, F63(2), F67} on the Poisson boundary and its applications to group theory. In general, a $\mu$--stationary system need not be proximal. How to relate general $\mu$--stationary systems to proximal $\mu$--invariant systems was studied by Furstenberg and Glasner \cite{F_G}. The main result of this part of our paper says that every $\mu$--injective system is proximal under some quite natural assumptions. 

\begin{thm}\label{thm: prox intro}
If $(\Gamma, \mu)$ is $\mu$--injective and satisfies \ref{A+}, then every $\mu$--invariant Borel probability measure on $[0,1]$ is $\mu$--proximal.
\end{thm}
The condition in the theorem is given as follows:
\begin{align*}\label{A} 
\textit{$\forall x\in (0, 1)$  $\exists g_0, g_1\in \Gamma$ with $g_0(0)=0, g_0(1)\in (0, 1)$\qquad\qquad\quad} \tag{$\mathfrak{A}$} \\
\textit{and $g_1(1)=1, g_1(0)\in (0, 1)$ such that  $g_0(x)\in (0,x)$ and $g_1(x)\in (x,1)$.
}
\end{align*}
\begin{align*}\label{A+} 
\textit{Condition \ref{A} and one of the following holds:
} \tag{$\mathfrak{A}+$} \\
 \mu\big(\{g\in\Gamma\, \colon \, \exists \varepsilon_{g}>0, \ \forall x\in (0,\varepsilon_g):\ g(x)<x\}\big)>0 \text{ or }\\
 \mu\big(\{g\in\Gamma\, \colon \,\exists \varepsilon_{g}>0,\ \forall x\in (1-\varepsilon_g,1):\ g(x)>x\}\big)>0.
 \end{align*}
The latter says that there is a non-trivial portion of maps that are \textit{below the diagonal in the neighbourhood of $0$} or \textit{above the diagonal in the neighbourhood of $1$}. There are plenty of examples of systems satisfying \ref{A+},  see for instance Example~\ref{ex: 3maps}.

Moreover, in $\mu$--injective systems, if every $\mu$--invariant probability measure is $\mu$--proximal, then there can be at most one such probability measure (Lemma~\ref{lem:all primal implies unique}). Thus we deduce the following generalization of \cite{BOS} concerning unique ergodicity.
\begin{cor}\label{cor: unique intro} If $(\Gamma, \mu)$ is $\mu$--injective and satisfies \ref{A+}, then there exists exactly one $\mu$--invariant Borel probability measure on $[0,1]$.
\end{cor}
A crucial ingredient in the proof of this statement is atomlessness of the measure. For invertible systems it is well known that an ergodic $\mu$--invariant Borel probability measure on $[0,1]$ is either finitely supported or atomless. Under the assumption of $\mu$--injectivity, we can generalize this to non-invertible, continuous, piecewise monotone maps. 
\begin{prop}\label{prop: inv set intro}
Assume that $(\Gamma,\mu)$ is $\mu$--injective. If  $\nu$ is a $\mu$--invariant Borel probability measure on $[0, 1]$, then the set of $\nu$--atoms of maximal $\nu$--value is a $\Gamma$--invariant set $\mu$--almost surely.
\end{prop}
Under the additional assumption of \ref{A} we deduce atomlessness.
\begin{thm}\label{thm: atomless intro}
If $(\Gamma,\mu)$ is $\mu$--injective and satisfies \ref{A}, then every $\mu$-invariant Borel probability measure on $[0,1]$ is atomless.
\end{thm}
(One might wonder why we do not need any ergodicity assumption in the statement above, but the truth is that the conditions already force the measure to be ergodic by uniqueness from Corollary~\ref{cor: unique intro}.)

Furthermore, the above enables us to deduce that the corresponding Markov operator is asymptotically stable (Definition~\ref{def: as stab}). 

\begin{thm}\label{thm: stab intro}
If $(\Gamma, \mu)$ is $\mu$--injective and satisfies \ref{A+}, then the corresponding Markov operator is asymptotically stable.
\end{thm}

In the last part of the paper, we are concerned with the quenched central limit theorem. Lately, quenched central limit theorems have been proved for various non-stationary Markov processes in \cite{Czudek_Szarek, SGHZ, Komorowski-book, Komorowski-Walczuk, Stenflo}. In this paper, we show this for $\mu$--injective random systems.
Our proof is based on the Gordin and Lif\v{s}ic theorem (see Section IV.8 in \cite{B_I}, see also \cite{D_L}) and some generalization of the \L uczy\'nska and Szarek recent results (see \cite{G_S}). As a by-product of this generalization we obtain that the system under consideration satisfies the synchronization condition. Combining this fact with the Gordin and Lif\v{s}ic theorem we obtain the quenched central limit theorem.

\begin{thm}\label{thm: CLT intro}
Let $(\Gamma,\mu)$ be $\mu$--injective and satisfy \ref{A+}. Let $(\xi_n)_{n\ge 1}$ be the corresponding stationary Markov chain with initial distribution $\nu$.
Then for any centered Lipschitz function $\varphi: [0, 1]\to\mathbb R$, i.e., $\int_{[0, 1]}\varphi\d\nu=0$, the random process $(\varphi(\xi_n))_{n\ge 1}$ satisfies the central limit theorem.
Moreover, the same is true for the process $(\varphi(\xi^x_n))_{n\ge 1}$, where $(\xi^x_n)_{n\ge 1}$ is the corresponding Markov chain starting from an arbitrary point $x\in [0, 1]$.
\end{thm}

The outline of the paper is as follows. Section~\ref{sec: prel} contains some notation and
definitions from the theory of random maps and Section~\ref{sec: mu-inj} introduces $\mu$--injective systems. In Section~\ref{Furs_entropy} we recall the Furstenberg entropy and some results on local contractivity obtained in our previous paper.
In Section~\ref{sec: atomless} we prove that every $\mu$--invariant measure for a $\mu$--injective system with \ref{A} has no atoms. In Section~\ref{sec: prox} we study proximality. Here we prove the main result: Theorem~\ref{thm: prox intro} (=Theorem~\ref{thm_21_05_24}) says that $\mu$--injective systems with \ref{A+} are proximal in the sense of Furstenberg. The theorem is preceded by some lemmas. Section~\ref{sec: stab} is devoted to asymptotic stability of the corresponding Markov operator. The main theorem of this part of the paper (Theorem~\ref{thm: stab intro}=Theorem~\ref{thm}) says that any $\mu$--injective system with \ref{A+} has an asymptotically stable corresponding Markov operator. Section~\ref{Synchronization} is devoted to the proof of synchronization which is used in proving  the quenched central limit theorem.
Finally, in Section~\ref{sec: clt} we prove the quenched central limit theorem for a Markov chain corresponding to $\mu$--injective systems with \ref{A+}.

\section{Acknowledgement}
The authors would like to express their gratitude to the anonymous reviewer for his careful and dedicated reading of the paper and for providing useful advice that improved the accessibility of the content. In particular, his/her bringing attention to paper \cite{G_S}, coauthored by one of us but completely forgotten,  allowed us to prove synchronization and make the central limit theorem section more concise and readable. 

\section{Preliminaries}\label{sec: prel}
\renewcommand{\thethm}{\thesection.\arabic{thm}}

By $\mathcal B([0,1])$ we denote the space of all Borel subsets of $[0, 1]$. By $\mathcal M$ and $\mathcal P$  we denote the space of all Borel measures and the space of all Borel probability measures on $[0, 1]$, respectively. By $C([0, 1])$ we denote the family of all bounded continuous functions equipped with the supremum norm $\|\cdot \|$. We shall write $\langle \nu, f\rangle$ for $\int_{[0, 1]} f\d\nu$.

In the space $\mathcal M$ we introduce the Wasserstein metric
$$
d_{W}(\nu_1, \nu_2)=\sup_{f\in F}  |\langle \nu_1, f\rangle-\langle \nu_2, f\rangle|,
$$
where $F\subset C([0, 1])$ is the set of all $f: [0, 1]\to\mathbb R$ such that
 $|f(x)-f(y)|\le |x-y|$ for all $x, y\in [0, 1]$. It is well known that the space $\mathcal P$ equipped with the Wasserstein metric is complete and the convergence
$$
\lim_{n\to\infty} d_W(\nu_n, \nu)=0\qquad\text{for $\nu_n, \nu\mathcal\,\in \mathcal P$}
$$
 is equivalent to the weak convergence of measures:
 $$
\lim_{n\to\infty}\langle \nu_n, f\rangle =\langle \nu, f\rangle\quad\text{for all $f\in C([0, 1])$.}
 $$

An operator $P : \mathcal{M}\rightarrow\mathcal{M}$ is called a {\it Markov operator} if it satisfies the following two conditions:
\begin{itemize}
\item $P(\lambda_1\nu_1+\lambda_2\nu_2)=\lambda_1P\nu_1+\lambda_2P\nu_2$ for $\lambda_1, \lambda_2\geq0, \nu_1, \nu_2\in\mathcal{M}$,
\item $P\nu\in\mathcal P$ for $\nu\in\mathcal{P}$.
\end{itemize}

A Markov operator $P$ is called a {\it Feller operator} if there is a linear operator $U : C([0, 1])\rightarrow C([0, 1])$ such that $U^*=P$, i.e.,
$$
\langle \nu, Uf\rangle=\langle P\nu, f\rangle\qquad \textrm{for $f\in C([0, 1]),\ \nu\in\mathcal{M}$}.
$$

A measure $\nu$ is called {\it $P$--invariant} for a Markov operator $P$ if $P\nu=\nu$. Since $[0, 1]$ is a compact metric space, every Feller operator $P$ has an invariant probability measure. For example, let $\nu_0\in\mathcal P$ and define $\hat\nu\in C([0, 1])^*$ by $\hat\nu(f)=\LIM(\frac{1}{n}\sum_{k=1}^n\langle P^k\nu_0, f\rangle)$, where $\LIM$ denotes a Banach limit. By the Riesz Representation Theorem $\hat\nu(f)=\langle \nu, f\rangle$, where $\nu\in\mathcal P$ is invariant.

\begin{definition}\label{def: as stab}
An operator $P$ is called {\it asymptotically stable} if there exists a $P$--invariant measure $\nu\in\mathcal{P}$ such that the sequence $(P^n\eta)_{n\ge 1}$ converges in the weak-$*$ topology to $\nu$ for any $\eta\in\mathcal{P}$, i.e.,
$$
\lim_{n\to\infty} \langle P^n\eta, f\rangle=\langle \nu, f\rangle
\qquad\text{for any $f\in C([0, 1])$.}
$$
\end{definition}
Necessarily, $\nu$ is then unique.

Let $\Gamma$ denote a family of continuous and piecewise monotone functions on $[0, 1]$ equipped with the topology induced by $\bigl( C([0,1]),\|\cdot\|\bigr)$. Let $\mu$ be a Borel probability measure on $\Gamma$. 
There is no loss of generality in assuming that $\supp\mu=\Gamma$. Otherwise, we may restrict our consideration to $\supp\mu$. The pair $(\Gamma, \mu)$ will be called a {\it stochastic dynamical system}.

The operator $P:\mathcal M\to\mathcal M$ defined by the formula
\begin{equation}\label{e1_31.08.24}
P\nu(A)=\int_{\Gamma} \nu(g^{-1}(A))\, \d\mu(g)\quad \text{for every $A\in\mathcal{B}([0,1])$ and $\nu\in\mathcal M$}
\end{equation}
is a Feller operator. In particular, we will be interested in whether this Markov operator is asymptotically stable in the above sense.

A Borel probability measure $\nu$ on $[0,1]$ is called \textit{$\mu$--invariant} (also known as \textit{$\mu$--stationary}) if $\nu$ is $P$--invariant, i.e., 
$$
\nu(A)=\int_{\Gamma} \nu(g^{-1}A)\, \d\mu(g) \qquad \text{ for every $A\in\mathcal{B}([0,1])$}.
$$

 Finally, let us recall the definition and some properties of the variation of a function.
Let $f:[0,1]\to\R$. By $\bigvee_0^1 f$ we denote the variation of $f$ over $[0,1]$, i.e.,
\[
\bigvee_0^1 f\ :=\ \sup \sum_{i=0}^n \bigl| f(x_i)-f(x_{i-1})\bigr|,
\]
where the supremum is taken over all partitions $0=x_0<x_1<\dots<x_n=1$. If $E\subset \R$, then
\[
|E|\ =\ \mathrm{diam}(E) \ :=\ \sup\big\{ |x-y|\colon x,y\in E\bigr\}.
\]
If $E_1, E_2$ are subsets of $[0,1]$, we write $E_1\leq E_2$ when $x_1\leq x_2$ for all $x_1\in E_1$ and $x_2\in E_2$. Analogously, we use the notation $E_1<E_2$.

The following observation will be central in estimations later on.
\begin{lem}\label{lem:BV-estimate diameter}
    Let $f:[0,1]\to\R$ and let $B_1,\dots, B_k\subset [0,1]$ be such that $B_1\leq B_2\leq \dots\leq B_k$. Then
    \[
        \sum_{i=1}^k\;\bigl| f(B_i)\bigr|\ \leq\ \bigvee_0^1 f.
    \]
\end{lem}
\begin{proof}
    If $B_i$ is a singleton, then $|f(B_i)|=0$. So, without loss of generality we may assume that each $B_i$ consists of at least two points. Pick $x_{i,1}, x_{i,2}$ in $B_i$, with $x_{i,1}<x_{i,2}$. Since $B_1\leq B_2\leq\dots \leq B_k$,
    \[
        \sum_{i=1}^k \bigl| f(x_{i,1})-f(x_{i,2})\bigr|\ \leq\ \bigvee_0^1 f.
    \]
    Taking the supremum over all possible choices for $x_{i,1}, x_{i,2}$, $i=1,\dots, k$, gives the result.
\end{proof}
\vskip 0.2cm

\section{Introduction to $\mu$--injective systems}\label{sec: mu-inj}
Let $(\Gamma, \mu)$ be a stochastic dynamical system. 
 We define $\mu$--injectivity. Heuristically, $\mu$--injectivity can be thought of as being injective ``on average'' with respect to the measure $\mu$. More precisely, let $\# g^{-1}(\{x\})$ denote the number of elements in $g^{-1}(x)$.
 \begin{definition}\label{def: mu-inj}
 We say that a stochastic dynamical system $(\Gamma,\mu)$ is \textit{$\mu$--injective}  if  
 $$ 
 \int_{\Gamma} \# g^{-1}(\{x\}) \, \d \mu(g) \leq  1 \text{ for all $x\in [0,1]$}.
 $$
\end{definition}

To illustrate $\mu$--injectivity, let us revisit an example of our paper \cite{BOS}, which contains a system  $\mu$--injective for certain choices of $\mu$, but not $\mu$--injective for some other choices of $\mu$.
 \begin{ex}\label{ex: 3maps} Let us consider the following three continuous functions.
     $$
g_1(x):=
\begin{cases}
3x&\text{ if } x\in[0,\frac{1}{3}],\\
-3x+2 &\text{ if } x\in(\frac{1}{3}, \frac{2}{3}],\\
3x-2 &\text{ if } x\in(\frac{2}{3}, 1],
\end{cases} \quad \ \ \ g_2(x):=\frac{x}{3}   \quad\text{ and } \quad g_3(x):=\frac{x}{3}+\frac{2}{3}.
$$

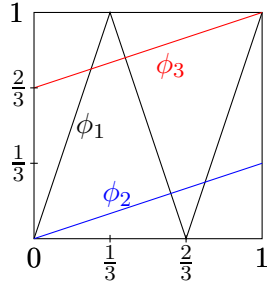
\begin{center}
\begin{tikzpicture}
\draw (0,0) rectangle (3,3);
\draw (0,0) -- (1,3) -- (2,0) -- (3,3);
\draw (0,0) node[anchor=north]{0};
\draw (3,0) node[anchor=north]{1};
\draw (0,3) node[anchor=east]{1};
\draw (1,-0.05)--(1,0.05) node[anchor=north]{$\frac{1}{3}$};
\draw (2,-0.05)--(2,0.05) node[anchor=north]{$\frac{2}{3}$};
\draw (0.75,1.8) node[anchor=north]{$g_1$};
\draw[blue] (0,0) -- (3,1);
\draw (-0.05,1)--(0.05,1) node[anchor=east]{$\frac{1}{3}$};
\draw (1.1,0.9) node[anchor=north]{\begin{color}{blue}$g_2$\end{color}};
\draw[red] (0,2) -- (3,3) ;
\draw (0,0) node[anchor=north]{0};
\draw (3,0) node[anchor=north]{1};
\draw (0,3) node[anchor=east]{1};
\draw (-0.05,2)--(0.05,2) node[anchor=east]{$\frac{2}{3}$};
\draw (1.8,2.6) node[anchor=north]{\begin{color}{red}$g_3$\end{color}};
\end{tikzpicture}
\end{center}
Let us take $\Gamma=\{g_1, g_2, g_3\}$ and $\mu(\{g_2\})=\mu(\{g_3\})=:p$. Since $\# g_1^{-1}(\{x\})=3$ { for $0<x<1$}, we see that the system is $\mu$--injective for the case that $p\geq \frac{2}{5}$ and not $\mu$--injective otherwise.\\
 Moreover, this system has additional features, namely it satisfies condition~\ref{A+}.
 \end{ex}

\section{The Furstenberg entropy and local contractivity}\label{Furs_entropy}

Let $(\Gamma, \mu)$ be a $\mu$--injective stochastic dynamical system. Since all functions $g$ in $\Gamma$ are continuous and piecewise monotone (with finitely many pieces), the set function $g^{-1}\nu(\cdot):=\nu (g(\cdot))$ is locally a measure and we can define the \textit{generalized Radon--Nikodym derivative} (for $\nu$-almost every $x$) by  
$$ 
\d_{\nu}g(x)=\frac{\d g^{-1}\nu}{\d\nu}(x):=\lim_{r\to 0}\frac{\nu(g(B_{r}(x)))}{\nu(B_{r}(x))} \quad \text{for $\nu$--a.e. } x.
$$
If the measure $\nu$ has no atoms, this limit is $\nu$--almost everywhere well--defined (see \cite{BOS} for details).

If $\ln^+ \d_\nu g(x)$ is $\nu\times\mu$-integrable, we can define the \textit{(generalized) Furstenberg entropy} as
$$ 
h_{\mu}(\nu):=-  \int_{[0,1]} \int_{\Gamma} \ln (\d_\nu g(x))\  \d\mu(g) \d\nu(x)\in (-\infty,+\infty].
$$ 
For any $x\in (0,1)$ we denote  
\[
\mathcal J^x:=\{I: \text{ $I\subset [0, 1]$ is a closed interval and}\,\,  x\in \mathrm{int}(I)\}
\]
and
\[
J(x,g):=\sup\left\{\frac{\nu(g (I))}{\nu(I)}\, : \, I\in\mathcal J^x \right\}\in[0,+\infty]
\]
with the convention $\frac{0}{0}=0$.  Observe that by definition $\d_\nu g(x)\leq J(x, g)$. Thus if $\ln^+J$ is $\nu\times\mu$--integrable then the entropy is well--defined. However, in Lemma 3.1 in \cite{BOS} we proved that if $\nu$ has no atoms and the system $(\Gamma, \mu)$ is $\mu$--injective, then $\ln^+J\in L^1(\Gamma\times[0,1],\mu\times\nu)$ (and $h_\mu(\nu)$ is well-defined). We also proved:

\begin{prop}[Proposition 3.3 in \cite{BOS}] \label{prop: h pos} Let $(\Gamma,\mu)$ be a stochastic dynamical system, and let $\nu$ be an atomless Borel probability measure on $[0,1]$. If  $(\Gamma,\mu)$ is $\mu$--injective, then either $h_{\mu}(\nu)>0$ (possibly infinite) or 
$\d_{\nu} g(x)\equiv 1$ for $\nu$--a.e. $x\in [0, 1]$ and  $\mu$--a.e. $g\in\Gamma$.
\end{prop}

To exclude the possibility that  $\d_{\nu} g(x)\equiv 1$ for $\nu$--a.e. $x\in [0, 1]$ and  $\mu$--a.e. $g\in\Gamma$ it is enough to assume that the stochastic dynamical system is contracting in the neighbourhood of some point $x_0\in \supp\nu$, i.e., there exists $\varepsilon>0$ such that the set of $g\in\Gamma$ satisfying:
$$
  g(x_0)=x_0
  $$
 and 
 $$
   |g(x)-g(x_0)|<|x-x_0|\quad\text{for all}\,\, x\in [x_0-\varepsilon, x_0+\varepsilon]\cap[0,1]\setminus\{x_0\}
$$
has positive $\mu$--measure. Then the condition 
$\d_{\nu} g(x)\equiv 1$ for $\nu$--a.e. $x\in [0, 1]$ and  $\mu$--a.e. $g\in\Gamma$ does not hold (see Theorem 3.6 in \cite{BOS}).

\begin{prop}[Proposition 2.3 in \cite{BOS}]\label{p_1_03.05.25}
Let $(\Gamma, \mu)$ be a stochastic dynamical system, and let $\nu$ be a $\mu$--invariant atomless ergodic Borel probability measure on the interval $[0,1]$ such that $\ln^+ J$ is $\nu\times\mu$--integrable and  
	$h_{\mu}(\nu)>0$.
	Then for every ${h}\in (0, h_{\mu}(\nu))$ and
	$\mu^{\otimes\mathbb{N}}$-a.e. $\omega=(g_1, g_2, \ldots)\in \Gamma^{\mathbb{N}}$ and $\nu$--a.e. $x\in [0,1]$ there exists a closed interval $I=I(\omega,x)$ with $x\in \mathrm{int}(I)$ such that
	$$
	|g_n\circ \ldots \circ g_1(I)|\leq \exp(-n\cdot {h})\quad \text{for all $n\ge 1$}.
	$$
\end{prop}

\section{Measures with no atoms}\label{sec: atomless}

Note that a stochastic dynamical system on $[0,1]$ has at least one $\mu$--invariant Borel probability measure, because $[0,1]$ is compact and the associated Markov operator $P$ is Feller, as discussed in Section~\ref{sec: prel}. 

\begin{prop}[Proposition~\ref{prop: inv set intro}]\label{prop: inv}
Assume that $(\Gamma,\mu)$ is $\mu$--injective. Let $\nu$ be a $\mu$--invariant probability measure on $[0, 1]$. Then the set of $\nu$--atoms of maximal value w.r.t. $\nu$ is $\Gamma$--invariant $\mu$--almost surely.
\end{prop}

\begin{proof}
If the set of atoms of $\nu$ is empty, then the above statement is trivial. So, assume the set of atoms of $\nu$ is non-empty. It is at most countable and there are finitely many atoms with maximal $\nu$-value. Set $\theta:= \max_{x\in [0, 1]}\nu(\{x\})$ and let $\{a_1, \ldots, a_N\}$ be the set of all $\nu$--atoms with maximal value $\theta$ under $\nu$.
Obviously, we have
$$
\nu(\{a_i\})=\int_{\Gamma} \nu(g^{-1}(a_i)) \, \d \mu(g)\qquad\text{for $i=1,\ldots, N$}.
$$
If for a $\mu$--positive set $E$ of $g$ we have $g^{-1}(a_i)\nsubseteq\{a_1, \ldots, a_N\}$,
then we would obtain 
\begin{align*}
\theta=\nu(\{a_i\})&=\int_{\Gamma} \nu(g^{-1}(a_i))\, \d \mu(g)\\
& = \theta\int_{\Gamma\setminus E}\# g^{-1}(a_i)\, \d \mu(g) + \int_E \nu(g^{-1}(a_i)) \, \d \mu(g)\\
& < \theta\int_{\Gamma\setminus E}\# g^{-1}(a_i)\, \d \mu(g) + \theta \int_E \# g^{-1}(a_i) \, \d \mu(g)\\
&=\theta \int_{\Gamma} { \# g^{-1}(a_i)} \, \d \mu(g) \leq  \theta,
\end{align*}
but this is impossible. 
Thus for $\mu$--a.e. $g$ we have $g^{-1}(a_i)\subset\{a_1, \ldots, a_N\}$, which proves the claim.
\end{proof}

Using the above we can proof the following:

\begin{thm}[Theorem~\ref{thm: atomless intro}]\label{thm: atomless}
If a stochastic dynamical system $(\Gamma,\mu)$ is $\mu$--injective and satisfies \ref{A}, then every $\mu$-invariant Borel probability measure on $[0,1]$ is atomless, and, in particular, it is concentrated on $(0, 1)$. 
\end{thm}

\begin{proof} Let $\nu$ be a $\mu$--invariant Borel probability measure on $[0,1]$.  Like in the proof of Proposition~\ref{prop: inv}, let
$\theta:= \max_{x\in [0, 1]}\nu(\{x\})$ and let  $\{a_1, \ldots, a_N\}$ be the set of all $\nu$-atoms with maximal $\nu$--value $\theta$ with 
$a_1<a_2<\cdots<a_N$ provided that $N>1$. 
Then, by Proposition~\ref{prop: inv} we know that $g^{-1}(a_i)\subset\{a_1, \ldots, a_N\}$ for  $\mu$--a.e. $g\in \Gamma$. Thus, we obtain that
$$
\begin{aligned}
N\theta&=\sum_{i=1}^N \nu(\{a_i\})=\int_{\Gamma}\sum_{i=1}^N \nu(g^{-1}(a_i)) \, \d \mu(g)\\
&=\int_{\Gamma}\theta \sum_{i=1}^N\sum_{j=1}^N {\bf 1}_{\{a_j\}}(g^{-1}(a_i))\, \d \mu(g)=
\int_{\Gamma}\theta \sum_{j=1}^N\sum_{i=1}^N {\bf 1}_{\{a_i\}}(g(a_j))\, \d \mu(g)\\
&=\theta\int_{\Gamma}\sum_{j=1}^N {\bf 1}_{\{a_1,\ldots, a_N\}}(g(a_j))\, \d \mu(g).
\end{aligned}
$$
This, in turn, implies that $\sum_{j=1}^N {\bf 1}_{\{a_1,\ldots, a_N\}}(g(a_j))=N$ for $\mu$--a.e. $g$. However, this is impossible since $\mu(\{g\in\Gamma: g(a_1)\in (0, a_1)\})>0$ if $a_1>0$ and $\mu(\{g\in\Gamma: g(a_2)\in (0, a_2)\})>0$ if $a_1=0$ and $N\ge 2$, due to condition~\ref{A} together with the assumption that $\supp\mu=\Gamma$. If $a_1=0$ and $N=1$, then
${\bf 1}_{\{a_1\}}(g(a_1))=0$ for some $g$, by (\ref{A}). This is also satisfied for all 
$h$ sufficiently close to $g$ which lie in 
$\supp(\mu)$.
  This completes the proof. 
 \end{proof}

The previous theorem implies also the existence of a $\mu$--invariant measure on the open interval $(0, 1)$.

\begin{remark}
Under the assumptions of Theorem~\ref{thm: atomless}
there exists a $\mu$--invariant probability measure on $(0, 1)$. This measure has no atoms.
\end{remark}

\section{Proximality}\label{sec: prox}

Let $\nu$ be a  $\mu$--invariant probability measure on $[0, 1]$. 
The function $h: \Gamma\to \mathbb R$ defined by the formula
$$
h(g)=\int_{[0, 1]} \varphi (g(x))\nu(\d x)\qquad\text{for $g\in\Gamma$}
$$
is a $\mu$--harmonic function on $\Gamma$ for any bounded Borel function $\varphi: [0, 1]\to\mathbb R$,  i.e., it satisfies the mean value property
$$
h(g)=\int_{\Gamma} h(g\circ g')\mu(\d g').
$$
It is well known that the sequence of measures $(\nu(g_n^{-1}\circ\cdots\circ g_1^{-1}(\cdot)))_{n\ge 1}$ is weakly convergent for $\mu^{\otimes\mathbb N}$--a.e. $\omega=(g_1, g_2,\ldots)\in\Gamma^{\mathbb N}$ (see \cite{F}, Corollary p. 20). Let us denote its limit by $\nu_{\omega}$. 
The measures $\nu$ will satisfy the barycenter 
equation
$$
\nu(\cdot)=\int_{\Gamma^{\mathbb N} }\nu_{\omega}(\cdot) \,\,\mu^{\otimes \mathbb N}(\d \omega).
$$

\begin{definition}[H. Furstenberg]\label{def: prox}
Following Furstenberg we say that a $\mu$--invariant measure  $\nu$ for a given stochastic dynamical system $(\Gamma,\mu)$ is $\mu$--{\it proximal} (or just \textit{proximal}) if there exists a measurable function $\Theta: \Gamma^{\mathbb N}\to [0, 1]$ such that
$$
\nu_{\omega}=\delta_{\Theta(\omega)}
$$
for $\mu^{\otimes \mathbb N}$--a.e.
$\omega=(g_1, g_2, \ldots)\in\Gamma^{\mathbb N}$.	
\end{definition}

A first observation concerns the relation between unique ergodicity and proximality.
\begin{lem}\label{lem:all primal implies unique}
    Let $(\Gamma,\mu)$ be a stochastic dynamical system such that every $\mu$--invariant measure is $\mu$--proximal. Then $(\Gamma,\mu)$ has at most one $\mu$--invariant measure.
\end{lem}
\begin{proof} Assume, contrary to our claim, that there exist two different $\mu$--invariant measures
$\nu^1, \nu^2\in\mathcal P$. Then we have $\nu^i_{\omega}=\delta_{\Theta_i(\omega)}$, $i=1, 2$,
for $\mu^{\otimes \mathbb N}$--a.e. $\omega$ and some measurable functions $\Theta_1, \Theta_2: \Gamma^{\mathbb N}\to [0, 1]$.
Since $\nu^i=\int_{\Gamma^{\mathbb N}} \nu^i_{\omega}\,\d\mu^{\otimes\mathbb N}$ for $i=1, 2$ and $\nu^1\neq\nu^2$, we obtain 
\begin{equation}\label{e1_21.05.24}
\mu^{\otimes\mathbb N}(\{\omega\in\Gamma^{\mathbb N}: \Theta_1({\omega})\neq\Theta_2({\omega})\})>0.
\end{equation}
On the other hand, since $\nu^3=(\nu^1+\nu^2)/2$ is $\mu$--invariant, it is $\mu$-proximal by assumption. Then $\nu^3_{\omega}=\delta_{\Theta_3(\omega)}$ 
for $\mu^{\otimes \mathbb N}$--a.e. $\omega$ and some measurable functions $\Theta_3: \Gamma^{\mathbb N}\to [0, 1]$.
The easy observation that
$$
\nu^3_{\omega}=\frac{\nu^1_{\omega}+\nu^2_{\omega}}{2}=\frac{\delta_{\Theta_1(\omega)}+\delta_{\Theta_2(\omega)}}{2}\neq\delta_{\Theta_3(\omega)}
$$
on some  $\mu^{\otimes\mathbb N}$--positive subset of $\omega\in\Gamma^{\otimes\mathbb N}$, by
(\ref{e1_21.05.24}), 
leads to a contradiction and completes the proof.
\end{proof}

\vskip2mm
Let us denote $\Gamma_0:=\{g\in\Gamma\colon g(0)=0\}$ and $
\Gamma_1:=\{g\in\Gamma\colon g(1)=1\}$.
The following lemma will be used repeatedly:

\begin{lem}\label{r1_02.10.25}  Let a stochastic dynamical system $(\Gamma,\mu)$ satisfy \ref{A}. Then for every $x\in [0, 1]$ and $\varepsilon>0$ there exists $(g_1,\ldots, g_n)\in\Gamma_0^n$ ({\it resp}. $(g_1,\ldots, g_n)\in\Gamma_1^n$)
for some $n\in\mathbb N$ such that $g_1\circ\cdots\circ g_n(x)\in [0, \varepsilon)$ ({\it resp}. $g_1\circ\cdots\circ g_n(x)\in (1-\varepsilon, 1]$). 
\end{lem}

\begin{proof} Let $x>0$ be given. Let $a_0$ be  the infimum of $a$ such that the point $x$ can be sent by a composition $g_1\circ\cdots\circ g_n$, $g_i\in\Gamma_0$; to a point in $[0, a)$. If $a_0>0$ we choose $g\in \Gamma_0$ such that $g(a_0)<a_0$, by \ref{A}, and taking the composition of $g$ with a suitable $g_1\circ\cdots\circ g_n$ at $x$ we obtain a contradiction with the definition of $a_0$. Hence for every $\varepsilon>0$ there exists $(g_1,\ldots, g_n)\in\Gamma_0^n$ 
for some $n\in\mathbb N$ such that $g_1\circ\cdots\circ g_n(x)\in [0, \varepsilon)$. 
The case for the interval $(1-\varepsilon, 1]$ may be proved analogously.
\end{proof}

The previous lemma may be slightly strengthened. Namely, we have:

\begin{lem}\label{l1_04.10.25}
Let a stochastic dynamical system $(\Gamma,\mu)$ satisfy \ref{A+}. Then for every $\varepsilon>0$ and $x, y\in [0, 1]$ there exist $m\ge 1$ and a sequence $(g_1, \ldots, g_m)\in\Gamma_0^m$  such that
\[
g_1\circ\cdots\circ g_m (w)\in [0, \varepsilon)\qquad\text{for $w\in\{x, y\}$.}
\]
\end{lem}

\begin{proof} Let $\varepsilon>0$. There is no loss of generality in assuming that $\varepsilon<\varepsilon_0$, where
 $\varepsilon_0$ is such that 
\[
\mu(\{g\in\Gamma_0: g(z)<z\,  \text{ for }\, z\in (0, \varepsilon_0)\})>0,
\] 
by condition~\ref{A+}. Fix $x, y\in [0, 1]$. By Lemma \ref{r1_02.10.25}, the compactness of $[0, 1]$ and continuity of the maps from $\Gamma$ we find a covering $U_1, \ldots, U_p$ of $[0, 1]$ and sequences of maps 
\[
(g_{1, 1}, \ldots, g_{1, m_1}), \ldots, (g_{p, 1}, \ldots, g_{p, m_p}), \quad g_{i, j}\in\Gamma_0
\] 
such that 
\[
g_{k, 1}\circ\cdots\circ g_{k, m_k}(U_k)\subset [0, \varepsilon)\qquad\text{for $k=1,\ldots, p$}.
\]
Let 
$i\in\{1,\ldots, p\}$ be such that that $x\in U_i$.
 Since $g_{k, 1}\circ\cdots\circ g_{k, m_k} (0)=0$ for every $k\in\{1, \ldots, p\}$, by continuity there exists $\varepsilon_1>0$ such that
\begin{equation}\label{e2_04.10.25}
g_{k, 1}\circ\cdots\circ g_{k, m_k} ([0, \varepsilon_1))\subset [0, \varepsilon)
\end{equation}
for every $k\in\{1, \ldots, p\}$. 
Now we may take $g\in \Gamma_0$ and $r\ge 1$ such that
$g^r([0, \varepsilon))\subset [0, \varepsilon_1)$. Hence
\[
g_{k, 1}\circ\cdots\circ g_{k, m_k}\circ g^r \circ g_{i, 1}\circ\cdots\circ g_{i, m_i} (x)\subset [0, \varepsilon)\]
for every $k\in\{1, \ldots, p\}$, by (\ref{e2_04.10.25}). 

 Since $g^r\circ g_{i, 1}\circ\cdots\circ g_{i, m_i} (y)\in U_j$ for some $j\in\{1, \ldots, p\}$, we obtain
\[
g_{j, 1}\circ\cdots\circ g_{j, m_j}\circ g^r\circ g_{i, 1}\circ\cdots\circ g_{i, m_i} (v)\in [0, \varepsilon).
\]
Obviously, we have also
\[
g_{j, 1}\circ\cdots\circ g_{j, m_j}\circ g^r\circ g_{i, 1}\circ\cdots\circ g_{i, m_i} (x)\in [0, \varepsilon)
\]
and our lemma is proved.
\end{proof}

We start the analysis of stochastic dynamical systems that satisfy condition~\ref{A+} with the lemma saying that some iterations will contract a large set, having measure almost $1$, to a set with a small diameter. In the following, we write the case when the second condition of \ref{A+} is fulfilled in 
brackets. 

\begin{lem}\label{l1_18.02.24}  If a stochastic dynamical system $(\Gamma,\mu)$ satisfies \ref{A+}, and has a $\mu$--invariant probability measure $\nu$, then for every $\varepsilon>0$ there exists a Borel set $B\subset [0, 1]$ with $\nu(B)>1-\varepsilon$ and a sequence $(g_1,\ldots, g_n)\in\Gamma_0^n$ (or $(g_1,\ldots, g_n)\in\Gamma_1^n$)
for some $n\in\mathbb N$ such that $g_1\circ\cdots\circ g_n(B)\subset [0, \varepsilon)$ (or $g_1\circ\cdots\circ g_n(B)\subset (1-\varepsilon, 1]$). In particular, $0, 1\in\supp\nu$.
\end{lem}

\begin{proof} The idea of the proof is to use condition~\ref{A+} in order to find two sequences of functions, one which ensures that the part of $B$ which is already in $[0,\varepsilon)$ stays in $[0,\varepsilon)$, even after applying the second sequence, which is chosen so that the part in $[\varepsilon,1]$ moves towards the interval $[0,\varepsilon)$. This process can be iterated to get the result.

To be more precise, say that the stochastic dynamical system $(\Gamma,\mu)$ satisfies the first condition in \ref{A+}, i.e., it is ``below the diagonal close to 0'', and  let $\varepsilon_0>0$ be such that
\begin{equation}\label{e1_9.02.24}
\mu(\{g\in\Gamma\, : \forall x\in (0,\varepsilon_0)\,(g(x)<x)\})>0.
\end{equation}
 Let
$$
\gamma:=\sup\{\nu(B): B\in\mathcal B([0,1]) \,\wedge \, \exists 
 \,g_1,\ldots, g_n\in\Gamma_0\,\,
 (g_1\circ\cdots\circ g_n(B)\subset [0, \varepsilon_0))\}.
$$
If we prove that $\gamma=1$, the assertion will follow. Indeed, then for every $\varepsilon>0$ we may choose $B\in\mathcal B([0,1])$ with $\nu(B)>1-\varepsilon$ and $g_1,\ldots, g_n\in\Gamma_0$ such that $g_1\circ\cdots\circ g_n(B)\subset [0, \varepsilon_0)$. By (\ref{e1_9.02.24}), in turn, we find $g_0\in\Gamma_0$ and $m\in\mathbb N$ such that
$g_0^m([0, \varepsilon_0))\subset [0, \varepsilon)$. Consequently, we obtain $g_0^m\circ g_1\circ\cdots\circ g_n(B)\subset [0, \varepsilon)$ as desired.

Suppose, contrary to the claim, that $\gamma<1$.  For every $x\in [0, 1]$ we may choose $h_1,\ldots, h_n\in\Gamma_0$ such that
$h_1\circ\cdots\circ h_n(x)\in [0, \varepsilon_0)$, by Lemma \ref{r1_02.10.25}. Now by continuity we may find an open neighbourhood $U_x$ of $x$ such that $h_1\circ\cdots\circ h_n(w)\in [0, \varepsilon_0)$ for $w\in U_x$, and let $\eta_x<\varepsilon_0$ be an arbitrary positive constant such that $h_1\circ\cdots\circ h_n([0, \eta_x])\subset [0, \varepsilon_0]$. Having defined the open sets $U_x$ for $x\in [0, 1]$ we find a finite covering 
\begin{equation}\label{e1_12.02.24}
 U_{x_1}\cup\cdots\cup U_{x_p} \supset [\varepsilon_0, 1].
\end{equation}
Set $\eta:=\min_{1\le i\le p} \eta_{x_i}$. Choose $B\in \mathcal B([0,1])$ with $\nu(B)>\gamma -\tfrac{1-\gamma}{2p}$ and
$g_1,\ldots, g_n\in\Gamma_0$ such that $g_1\circ\cdots\circ g_n(B)\subset [0, \varepsilon_0)$, and let $g_0\in\Gamma_0$ be such that $g_0(x)<x$ for all $x\in [0, \varepsilon_0)$, by (\ref{e1_9.02.24}).
Now take $n_0\in\mathbb N$ such that $g_0^{n_0}([0, \varepsilon_0))\subset [0, \eta)$. By the definition of $\gamma$ we must have
\begin{equation}\label{e2_12.02.24}
\nu(\{x\in B': g_0^{n_0}\circ g_1\circ\cdots\circ g_n(x)\in 
[\varepsilon_0, 1]\})>\tfrac{1-\gamma}{2},
\end{equation}
where $B'=[0, 1]\setminus B$.  If not 
\[
\nu(\{x\in B': g_0^{n_0}\circ g_1\circ\cdots\circ g_n(x)\in 
[0, \varepsilon_0)\})\ge\tfrac{1-\gamma}{2}
\]
since $\nu (B')\ge1-\gamma$. Then for
\[
B_0:=B\cup \{x\in B': g_0^{n_0}\circ g_1\circ\cdots\circ g_n(x)\in 
[0, \varepsilon_0)\}
\]
we have
$$
\nu (B_0)>\gamma\quad\text{and}\quad
g_0^{n_0}\circ g_1\circ\cdots\circ g_n(B_0)\subset [0, \varepsilon_0),
$$
and this is contrary to the definition of $\gamma$. Further, by (\ref{e1_12.02.24}) and  (\ref{e2_12.02.24}) we obtain
$$
\begin{aligned}
&\sum_{i=1}^p \nu(\{x\in B': g_0^{n_0}\circ g_1\circ\cdots\circ g_n(x)\in U_{x_i}\})\\
&\ge \nu(\{x\in B': g_0^{n_0}\circ g_1\circ\cdots\circ g_n(x)\in \bigcup_{i=1}^p U_{x_i}\})\\
&\ge \nu(\{x\in B': g_0^{n_0}\circ g_1\circ\cdots\circ g_n(x)\in [\varepsilon_0, 1]\})\ge\tfrac{1-\gamma}{2}.
\end{aligned}
$$
Hence there exists $i\in\{1,\ldots, p\}$ such that
$$
\nu(\{x\in B': g_0^{n_0}\circ g_1\circ\cdots\circ g_n(x)\in U_{x_i}\})\ge\tfrac{1-\gamma}{2p}.
$$
Let
$$
\tilde B:=\{x\in B': g_0^{n_0}\circ g_1\circ\cdots\circ g_n(x)\in U_{x_i}\}.
$$
Choose $h_1,\ldots, h_{m}\in\Gamma_0$ such that
$h_1\circ\cdots\circ h_{m}(U_{x_i})\subset [0, \varepsilon_0)$ and observe that
$h_1\circ\cdots\circ h_{m}([0, \eta))\subset [0, \varepsilon_0)$. Therefore, 
$$
h_1\circ\cdots\circ h_{m}\circ g_0^{n_0}\circ g_1\circ\cdots\circ g_n(B\cup \tilde B)\subset [0, \varepsilon_0).
$$
Moreover, $\nu (B\cup \tilde B)>\gamma-\tfrac{1-\gamma}{2p}+\tfrac{1-\gamma}{2p}=\gamma$, which is impossible, by the definition of $\gamma$.  Thus $\gamma=1$ and the proof for the system $(\Gamma, \mu)$ which is below the diagonal close to 0 is completed. 

Finally, from what we have been proved and the continuity of $g\in \Gamma$ it follows that for every $\varepsilon>0$ there exists 
$n\in\mathbb N$, a compact set $K$ with $\nu(K)>0$ and an open set $U\subset \Gamma^n$ such that for every $(g_1,\ldots g_n)\in U$ we have $g_1\circ\cdots\circ g_n (K)\subset [0, \varepsilon)$. Hence
\[
\begin{aligned}
\nu([0, \varepsilon))=P^n\nu([0, \varepsilon))
&\ge \int_U {\bf 1}_{[0, \varepsilon)}(g_1\circ\cdots\circ g_n(x))\d\mu^{\otimes n}(g_1, \ldots, g_n)\nu(\d x)\\
&\ge \mu^{\otimes n}(U)\nu(K)>0,
\end{aligned}
\]
the last inequality by the fact that $\supp\mu=\Gamma$ and, consequently, $\mu^{\otimes n}(U)>0$.
Since $\varepsilon>0$ was arbitrary, we obtain $0\in\supp\nu$. The case when $(\Gamma, \mu)$ is above the diagonal close to 1 is proved analogously.
\end{proof}

\begin{rem}\label{r1_20.04.24} In the assertion of Lemma \ref{l1_18.02.24} we may require that the set $B$ is closed. In fact, for every $\varepsilon>0$ the closure of any set B found for $\varepsilon/2$ will meet the requirement.
\end{rem}

In the next lemma, closely related to Lemma \ref{l1_18.02.24}, 
 we prove that we may find a large set (with measure close to $1$) and as many different iterations as we wish such that they map the large set on disjoint, and therefore, small subsets. Similar ideas were used by Kudryashov in the study of bony attractors (see \cite{Ku}).

\begin{lem}\label{lem_07.03.24} Under the assumptions of Lemma \ref{l1_18.02.24} for every $\varepsilon >0$ and $n\in\mathbb N$ there exist $(g_{1, 1},\ldots, g_{1, i_1})\in \Gamma^{i_1},\ldots, (g_{n, 1},\ldots, g_{n, i_n})\in \Gamma^{i_n}$
and a closed set $B\subset [0, 1]$ with $\nu(B)>1-\varepsilon$ such that 
$$
g_{n, 1}\circ\cdots\circ g_{n, i_n}(B) <\, \cdots \, < g_{2, 1}\circ\cdots\circ g_{2, i_2}(B) < g_{1, 1}\circ\cdots\circ g_{1, i_1}(B).
$$ 
\end{lem}

\begin{proof} {The idea of the proof is obvious. By Lemma \ref{l1_18.02.24} almost everything of $B$ may be brought to a sufficiently small interval $[0, \varepsilon)$. Since the measure is atomless, it may be assumed that it is supported on some closed interval 
$[\varepsilon_0, \varepsilon]$. Now one can apply a lot of contraction  towards $0$, ensuring that the measure supported on  $[\varepsilon_0, \varepsilon)$ is brought to $[0, \varepsilon_0)$. Iterating this procedure would conclude the proof. }\\
Formally, fix an $\varepsilon>0$ and assume that the stochastic dynamical system $(\Gamma,\mu)$ is below the diagonal close to 0. The case when it is above the diagonal close to 1 may be done analogously.  We thus assume that 
\[
\mu(\{g\in\Gamma\, : \,g(x)<x,\forall x\in (0,\varepsilon)\})>0.
\]
 By Lemma \ref{l1_18.02.24} (see also Remark \ref{r1_20.04.24}) there exists a closed set $B\subset [0, 1]$ with $\nu(B)>1-\varepsilon$ and a sequence $(g_1,\ldots, g_n)\in\Gamma_0^n$, $n\in\mathbb N$, such that $g_1\circ\cdots\circ g_n(B)\subset [0, \varepsilon)$. By the continuity of the functions in $\Gamma$ there exists an open neighbourhood $U\subset \Gamma^n$ of $(g_1,\ldots, g_n)$ such that $h_1\circ\cdots\circ h_n(B)\subset [0, \varepsilon)$ for every $(h_1,\ldots, h_n)\in U$.
Obviously, $\mu^{\otimes n} (U)>0$. Since $\nu(\{0\})=P^n\nu(\{0\})=0$, by Theorem~\ref{thm: atomless}, there exists $({\tilde g}_1,\ldots, {\tilde g}_n)\in U$ such that $\nu ( {\tilde g_n}^{-1}\circ\cdots\circ {\tilde g_1}^{-1}(0)\cap B)=0$. Therefore, for sufficiently small $\eta>0$ and $B'=B\setminus {\tilde g_n}^{-1}\circ\cdots\circ {\tilde g_1}^{-1}([0, \eta))$ we have $\nu(B')>1-\varepsilon$. Taking $h\in\Gamma$ such that $0<h(x)<x$ for $x\in (0, \varepsilon)$ we will find an increasing sequence $k_1, \ldots, k_n\in\mathbb N$ such that the sets
$$
h^{k_i}\circ {\tilde g}_1\circ\cdots\circ {\tilde g}_n(B')
$$
for $i=1,\ldots, n$ are pairwise disjoint. This completes the proof.
\end{proof}

In our further consideration the crucial role will be played by 
the variation of compositions of the functions from $\Gamma$. 

\begin{lem}\label{lem: bnd var}
Let $\Gamma$ be a collection of piecewise monotone continuous functions on $[0,1]$, and let $\mu$ be a Borel probability measure on $\Gamma$ with $\supp\mu=\Gamma$.
If $\nu$ is a $\mu$--invariant measure and $(\Gamma, \mu)$ is $\mu$--injective, 
 then for $\mu^{\otimes \mathbb N}$--almost every $\omega=(g_1,\ldots, g_n,\ldots)$ we have
 $$
 \sup_{n\in\mathbb N} \bigvee_0^1 g_1\circ\cdots\circ g_n<+\infty.
 $$
\end{lem}

\begin{proof} {The proof is based on the observation that $\mu$--injectivity implies that the variation of random compositions forms a supermartingale. Then 
one could apply the Martingale Convergence Theorem to complete the proof.}
\\ Indeed, let ${\bf g_1}, {\bf g_2},\ldots$ denote the independent random variables defined on $(\Gamma^{\mathbb N}, \mu^{\otimes\mathbb N})$ with values in $\Gamma$ and distribution $\mu$, i.e., 
${\bf g_i}(\omega)=g_i$ for $\omega=(g_1, g_2,\ldots)$.
 We set
$$
\mathcal N_n (\cdot , x):=\#\{z\in [0, 1]:  {\bf g_1}(\cdot) \circ\cdots\circ {\bf g_n}(\cdot) (z)=x\}, 
$$
and let
$$
V_n(\omega)=\int_{[0, 1]}\mathcal N_n (\omega , x) \d x \quad\text{for $n\in\mathbb N$ and $\omega=(g_1, g_2,\ldots)$}.
$$

Now we are going to show that $(V_n)_{n\ge 1}$ is a positive supermartingale with respect to the filtration $\mathcal F_n=\sigma({\bf g_1},\ldots, {\bf g_n})$, $n\in\mathbb N$.
Namely, by $\mu$--injectivity and the Fubini theorem for every $\omega\in \Gamma^{\mathbb N}$ we have
$$
\begin{aligned}
\mathbb E [V_{n+1} \mid\mathcal F_n ](\omega)&=\int_{[0, 1]}\int_{\Gamma} \sum_{z\in g_n^{-1}\circ\cdots\circ g_1^{-1}(x)} \#g_{n+1}^{-1}(\{ z\})\mu (\d g_{n+1})\d x\\
&=\int_{[0, 1]}\sum_{z\in g_n^{-1}\circ\cdots\circ  g_1^{-1}(x)}\int_{\Gamma} \# g_{n+1}^{-1}(\{ z\})\mu (\d g_{n+1}) \d x\\
&\le \int_{[0, 1]} \#\{z: z\in g_n^{-1}\circ\cdots\circ g_1^{-1}(x)\} \d x\\
&=
\int_{[0, 1]} \mathcal N_n (\omega , x)\d x=V_n(\omega).
\end{aligned}
$$
By the Martingale Convergence Theorem the random variables 
$V_n (\cdot)$ are $\mu^{\otimes\mathbb N}$--a.s. convergent to a  finite random variable. Hence for $\mu^{\otimes\mathbb N}$--almost every $\omega\in\Gamma^{\mathbb N}$ we have $\sup_{n\ge 1} V_n (\omega)<+\infty$.

On the other hand, by the Banach--Vitali theorem (see \cite{L})
$$
V_n(\omega)=\int_{[0, 1]}\mathcal N_n (\omega , x) \d x=\bigvee_0^1 g_1\circ\cdots\circ g_n
$$
for $\omega=(g_1, g_2,\ldots) \in\Gamma^{\mathbb N}$ and $n\in\mathbb N$,
and consequently $\sup_{n\ge 1} \bigvee_0^1 g_1\circ\cdots\circ g_n<+\infty$ for 
$\mu^{\otimes\mathbb N}$--almost every $\omega=(g_1, g_2,\ldots)\in\Gamma^{\mathbb N}$. This completes the proof.
\end{proof}

In the following theorem we prove that any $\mu$-invariant Borel probability measure of the considered system is  $\mu$--proximal.

\begin{thm}[Theorem~\ref{thm: prox intro}]\label{thm_21_05_24}
Let $\Gamma$ satisfy \ref{A}, and let a Borel probability measure $\mu$ with $\supp\mu=\Gamma$ be given. Assume that the stochastic dynamical system $(\Gamma,\mu)$ is $\mu$--injective and below the diagonal close to 0 (or above close to 1). If $\nu$ is a $\mu$--invariant Borel probability measure, then $\nu$ is $\mu$--proximal.
 \end{thm}

\begin{proof} {The proof is based on the observation that for almost every $g_1, \ldots, g_n$ one may find some family of sequences (of positive measure)  $g_{n+1}, \ldots, g_{n+k}$ such that the measure $\nu(g_{n+k}^{-1}\circ\cdots\circ g_1^{-1}(\cdot))$ is mostly supported on some $\varepsilon$--interval. This would be possible due to Lemma \ref{lem_07.03.24} and the fact that the variation of $g_1\circ\cdots\circ g_n$ are bounded, by Lemma \ref{lem: bnd var}. Passing with $n$ to $\infty$ and $\varepsilon$ to $0$ would complete the proof.}
\\
More precisely, let $\nu$ be a $\mu$--invariant probability measure. It is concentrated on $(0, 1)$, according to Theorem~\ref{thm: atomless}. Let $\Gamma_{\#}^{\mathbb N}\subset\Gamma^{\mathbb N}$ be such that $\mu^{\mathbb N}(\Gamma_{\#}^{\mathbb N})=1$ and
the sequence of measures $(\nu(g_n^{-1}\circ\cdots\circ g_1^{-1}(\cdot)))_{n\ge 1}$ weakly converges to $\nu_{\omega}$ for $\omega\in \Gamma_{\#}^{\mathbb N}$, i.e., for $\omega\in \Gamma_{\#}^{\mathbb N}$ it holds
\begin{equation}\label{e1_24.02.24}
\lim_{n\to\infty} \int_{[0, 1]} \varphi(g_{1}\circ\cdots\circ g_{n}(x))\nu(\d x)=\int_{[0, 1]} \varphi(x)\nu_{\omega}(\d x)
\quad\text{for every $\varphi\in C([0, 1])$}
\end{equation}
(see \cite{F}, Corollary p. 20).
We are going to verify the following claim (see \cite{Czudek_Szarek, Navas}):
\vskip2mm
{\bf Claim:} for any $\varepsilon>0$ there exists $\Gamma_{\varepsilon}^{\mathbb N}\subset\Gamma_{\#}^{\mathbb N}$ with $\mu^{\otimes\mathbb N}(\Gamma_{\varepsilon}^{\mathbb N})>1-\varepsilon$ satisfying the following property:  for every $\omega\in\Gamma_{\varepsilon}^{\mathbb N}$ there exists an interval $I$ of length $|I|\le\varepsilon$ such that $\nu_{\omega}(I)\ge 1-\varepsilon$. Hence we obtain that $\nu_{\omega}=\delta_{\Theta(\omega)}$ for all $\omega$ from some set $\hat\Gamma^{\mathbb N}$ with $\mu^{\otimes\mathbb N}(\hat\Gamma^{\mathbb N})=1$. Here  $\Theta(\omega)$ is a point from $[0, 1]$.

To see this take a sequence $(\varepsilon_n)_{n\ge 1}$, $\varepsilon_n\to 0$ and $\Gamma^{\mathbb N}_{\varepsilon_n}\subset\Gamma_{\#}^{\mathbb N}$ such that $\mu^{\otimes\mathbb N}(\Gamma_{\varepsilon_n}^{\mathbb N})>1-\varepsilon_n$ and for every $\omega\in \Gamma_{\varepsilon_n}^{\mathbb N}$ there exists an interval $I$ of length $|I|\le\varepsilon_n$ such that $\nu_{\omega}(I)\ge 1-\varepsilon_n$. Set $\hat\Gamma^{\mathbb N}=\bigcap_{n=1}^{\infty}\bigcup_{m=n}^{\infty}\Gamma_{\varepsilon_m}^{\mathbb N}$ and observe that $\mu^{\otimes\mathbb N}(\hat\Gamma^{\mathbb N})=1$. Moreover, for $\omega\in \hat\Gamma^{\mathbb N}$ there exists a subsequence $\varepsilon_{m_n}$ converging to $0$ and a sequence $I_n$ such that $|I_{n}|\to 0$
and $\nu_{\omega}(I_n)\to 1$. Since $I_n\subset [0, 1]$ and $[0, 1]$ is compact, passing to a subsequence we may assume that there exists $\Theta(\omega)\in [0, 1]$ such that $I_n\to \Theta(\omega)$. Then we obtain $\nu_{\omega}(\{\Theta(\omega)\})= 1$ and consequently $\nu_{\omega}=\delta_{\Theta(\omega)}$.

\vskip2mm
Fix an $\varepsilon>0$, and let $\tilde\Gamma_{\varepsilon}^{\mathbb N}$ be such that $\mu^{\otimes\mathbb N}(\tilde\Gamma_{\varepsilon}^{\mathbb N})>1-\varepsilon$ 
 and $\sup_{n\ge 1} V_n(\omega)<M$ for some $M>0$ and every $\omega\in \tilde\Gamma_{\varepsilon}^{\mathbb N}$ using Lemma~\ref{lem: bnd var}. Let $k\in\mathbb N$ be such that $M/k<\varepsilon$.

 By Lemma \ref{lem_07.03.24} 
 we find a closed set $B\subset [0, 1]$ with $\nu(B)>1-\varepsilon$ and $(g_{1, 1},\ldots, g_{1, i_1})\in \Gamma^{i_1},\ldots, (g_{k, 1},\ldots, g_{k, i_k})\in \Gamma^{i_k}$
such that one has 
$$
g_{k, 1}\circ\cdots\circ g_{k, i_k}(B)<\,\cdots\,
< g_{2, 1}\circ\cdots\circ g_{2, i_2}(B)< g_{1, 1}\circ\cdots\circ g_{1, i_1}(B)
$$ 
In particular, the sets are pairwise disjoint. By continuity, for every $j\in\{1, \ldots, k\}$ there exists an open neighbourhood $U_j\subset \Gamma^{i_j}$ of $(g_{j, 1},\ldots, g_{j, i_j})$ such that $\hat B_1, \ldots, \hat B_k$ are pairwise disjoint, where 
$$
\hat B_j=\bigcup_{(h_{j, 1},\ldots, h_{j, i_j})\in U_j} h_{j, 1}\circ\cdots\circ h_{j, i_j}(B)
\quad\text{for $j=1,\ldots, k$},
$$ 
while still $\hat{B}_k<\cdots <\hat{B}_2<\hat{B}_1$.
 We see that $\kappa:=\min_{1\le j\le k} \mu^{\otimes i_j}(U_j)>0$. 
 According to Lemma \ref{lem:BV-estimate diameter}, for every $\omega=(g_1,\ldots, g_m)\in\Gamma^{m}$,  $m\in\mathbb N$, there exists $j\in\{1, \ldots, k\}$ such that
 $$
 |g_1\circ\cdots\circ g_m(\hat B_j)|\le \bigvee_0^1 g_1\circ\cdots\circ g_m/k.
 $$
 Hence  
 $$
 |g_1\circ\cdots\circ g_m\circ h_{j, 1}\circ\cdots\circ h_{j, i_j}(B)|\le\bigvee_0^1 g_1\circ\cdots\circ g_m/k
$$
 for $(h_{j, 1},\ldots, h_{j, i_j})\in U_j$.
 
 Let $j_*=\max_{1\le j\le k} i_j$. This shows that for any cylinder in $\Gamma^{\mathbb N}$, defined by fixing the first entries $(g_1,\ldots, g_n)$, the conditional probability that $(g_{n+1}, \ldots, g_{n+j_*})$ are such that
 $$
 |g_{1}\circ\cdots\circ g_{{n+l}}(B)| \ge\bigvee_0^1 g_{i_1}\circ\cdots\circ g_{i_n}/k
 \quad\text{for all $l=1,\ldots, j_*$}
 $$
 is less than $1-\kappa$.

 Hence for $\mu^{\otimes\mathbb N}$--a.e. $\omega=(g_1, g_2, \ldots)\in\Gamma^{\mathbb N}$ there exist  increasing sequences $(m_n)_{n\ge 1}$ and $(\tilde m_n)_{n\ge 1}$ (depending on $\omega$)
 such that $0<m_n-\tilde m_n\le j_*$ and 
 $$
 |g_{1}\circ\cdots\circ g_{{m_n}}(B)| \le\bigvee_0^1 g_{1}\circ\cdots\circ g_{\tilde m_n}/k\quad\text{for all $n\ge 1$}.
 $$
 Hence there exists 
 $\tilde\Gamma_{\varepsilon}^{\mathbb N}\subset\Gamma_{\varepsilon}^{\mathbb N}$ such that still $\mu^{\otimes\mathbb N}(\tilde \Gamma_{\varepsilon}^{\mathbb N})>1-\varepsilon$ and 
 for $\omega=(g_1, g_2, \ldots)\in\tilde\Gamma_{\varepsilon}^{\mathbb N}$ there exists an increasing sequence $(m_n)_{n\ge 1}$ (depending on $\omega$)
 such that 
 \begin{equation}\label{e1_6.08.24}
 |g_{1}\circ\cdots\circ g_{{m_n}}(B)| \le M/k<\varepsilon\quad\text{for all $n\ge 1$}.
 \end{equation}
 
Now take $\omega=(g_1, g_2, \ldots)\in\tilde\Gamma_{\varepsilon}^{\mathbb N}$. Since there exists a subsequence $(m_n)_{n\ge 1}$ of positive integers such that (\ref{e1_6.08.24}) is satisfied, passing to a subsequence if necessary, we may assume that there exists a closed interval $I\subset [0, 1]$ with $|I|<\varepsilon$ such that
 $$
 g_{1}\circ\cdots\circ g_{m_n}(B)\subset I \quad\text{for all $n\ge 1$}.
 $$
 From (\ref{e1_24.02.24}) we know that
 the sequence of measures $(\nu(g_{m_n}^{-1}\circ\cdots\circ g_1^{-1}))_{n\ge 1}$ as 
 a subsequence of  $(\nu(g_{n}^{-1}\circ\cdots\circ g_1^{-1}))_{n\ge 1}$ converges weakly to $\nu_{\omega}$. Hence, by the Portmanteau Theorem, we obtain
 $\nu_{\omega} (I)\ge \nu(g_{m_n}^{-1}\circ\cdots\circ g_1^{-1}(I))\ge \nu(B)>1-\varepsilon$ and the proof is complete.
\end{proof}

From the above theorem we obtain the following:

\begin{cor}[Corollary~\ref{cor: unique intro}]\label{c1_21_05_24}
Assume that a stochastic dynamical system $(\Gamma, \mu)$ satisfies the assumptions of Theorem \ref{thm_21_05_24}. Then it is uniquely ergodic.
\end{cor}
\begin{proof} From Theorem \ref{thm_21_05_24} it follows that every $\mu$--invariant measure $\nu$ is $\mu$--proximal. 
By compactness, such a measure exists.
Lemma \ref{lem:all primal implies unique} yields that the $\mu$--invariant measure is unique. 
\end{proof}

\section{Local contractivity}

We can now state the analogue of Proposition 
\ref{p_1_03.05.25} for the studied system.

\begin{thm}\label{T_1_03.05.25} Let a stochastic dynamical system $(\Gamma,\mu)$ be $\mu$--injective and satisfy \ref{A+}. Then 
there exists a unique $\mu$--invariant measure $\nu$. 
Moreover, $h_{\mu}(\nu)>0$ and for every ${h}\in (0, h_{\mu}(\nu))$ and $\mu^{\mathbb{N}}$-a.e. $\omega=(g_1, g_2, \ldots)\in \Gamma^{\mathbb{N}}$ and $\nu$--a.e. $x\in [0,1]$ there exists a closed interval $I=I(\omega,x)$ with $x\in \mathrm{int}(I)$ such that
	$$
	|g_n\circ \ldots \circ g_1(I)|\leq \exp(-n\cdot {h})\quad \text{ for all $n\ge 1$}.
	$$
\end{thm}
\begin{proof} By Corollary \ref{c1_21_05_24} the system $(\Gamma, \mu)$ has a unique invariant measure $\nu$. Obviously, $\nu$ is ergodic. Moreover $\nu$ is atomless, by Theoerem \ref{thm: atomless}. From Proposition \ref{prop: h pos} it follows that either $h_{\mu}(\nu)>0$ (possibly infinite) or 
$\d_{\nu} g(x)\equiv 1$ for $\nu$--a.e. $x\in [0, 1]$ and  $\mu$--a.e. $g\in\Gamma$. Further, condition~\ref{A+} implies that $(\Gamma, \mu)$ is contracting in the neighbourhood of $0$ and $0\in\supp\nu$, by Lemma \ref{l1_18.02.24}. Hence the condition $\d_{\nu} g(x)\equiv 1$ for $\nu$--a.e. $x\in [0, 1]$ and  $\mu$--a.e. $g\in\Gamma$ does not hold, and consequently $h_{\mu}(\nu)>0$. Finally, from Lemma 3.1 in \cite{BOS} it follows that $\ln^+J\in L^1(\Gamma\times[0,1],\mu\times\nu)$ and all assumptions of Theorem \ref{p_1_03.05.25} are satisfied. We complete the proof by applying Theorem \ref{p_1_03.05.25}.
\end{proof}

\section{Stability}\label{sec: stab}

When the measure $\mu$ is supported on monotonic maps stability of the corresponding Markov operator is easily derived from proximality. In fact, we have the following proposition:

\begin{prop}\label{prop:monotone case asympt stab}
Let $\Gamma$ consist of strictly monotonic continuous functions and let $\mu$ be a Borel probability measure on $\Gamma$ such that $\Gamma=\supp\mu$. Assume that every $\mu$-invariant probability measure of the stochastic dynamical system $(\Gamma,\mu)$ is atomless. If $(\Gamma, \mu)$ has a $\mu$--invariant probability measure $\nu$ that is $\mu$--proximal and $\{0, 1\}\subset\supp\nu$, then the Markov operator associated to $(\Gamma,\mu)$ is asymptotically stable.
\end{prop}

\begin{proof} {Because the measure $\nu$ is $\mu$--proximal, all the functions are monotonic and $0, 1$ are in the support of $\nu$, one could prove that almost surely $(g_{1}\circ\cdots\circ g_{n}(x))_{n\ge 1}$ converges and its limit is independent of the starting point $x$. Hence one could conclude that the iterates $P^n\delta_x$ and $P^n\delta_y$ are close for every $x, y$. Observing that $P^n\tilde\nu=\int_{[0, 1]} P^n\delta_x\,\tilde\nu(\d x)$ for any probability measure $\tilde \nu$, one could obtain the proof.}
\\

It suffices to show that
\begin{equation}\label{e1_4.10.25}
\lim_{n\to\infty} \left|\langle P^n\nu_1, \varphi\rangle-
\langle P^n\nu_2, \varphi\rangle\right|=0
\end{equation}
for every $\nu_1, \nu_2\in\mathcal P$ and an arbitrary Lipschitz function $\varphi: [0, 1]\to\mathbb R$ due to density of the space all Lipschitz functions in  $C([0, 1])$. 

The $\mu$--proximality of $\nu$ yields a map $\Theta:\Gamma^\N\to [0,1]$ such that for $\mu^{\otimes\N }$--a.e. $\omega=(g_1,g_2,\dots)\in \Gamma^\N$, the sequence of measures $\nu(g_n^{-1}\circ\dots\circ g_1^{-1}(\cdot))$ converges weakly to $\delta_{\Theta(\omega)}$. Since $\nu$ is atomless,
\[
0 = \nu\bigl(\{0\}\bigr) = \int_{\Gamma^\N} \delta_{\Theta(\omega)}\bigl(\{0\}\bigr) \mu^{\otimes\N}(d\omega).
\]
Thus, $\Theta(\omega)\neq 0$ for $\mu^{\otimes\N}$-a.e. $\omega$. A similar argument applies to $\{1\}$. Thus, $\Theta(\omega)\not\in\{0,1\}$ for $\mu^{\otimes\N}$--a.e. $\omega$.

Fix $x, y\in (0, 1)$, $x<y$. Since $\nu$ has no atoms and $0,1\in\supp\nu$ we see that $\nu((0, x))>0$ and $\nu((y, 1))>0$. From proximality it follows that for $\mu^{\otimes \mathbb N}$--a.e. $\omega=(g_1, g_2,\ldots)$ and for every $\varepsilon>0$ sufficiently small such that $I_\omega^\varepsilon:=(\Theta(\omega)-\varepsilon,\Theta(\omega)+\varepsilon)\subset (0,1)$, there exists $N\in\mathbb N$ such that for $n\ge N$ we may find $x_n\in (0, x)$ and $y_n\in (y, 1)$ satisfying 
$g_1\circ \cdots\circ g_n(x_n),
g_1\circ \cdots\circ g_n(y_n)\in I_\omega^\varepsilon$. In fact, because $\nu$ is atomless, weak convergence implies 
\[
    \nu\bigl(g_n^{-1}\circ\dots\circ g_1^{-1}(I_\omega^\varepsilon)\bigr)\ \to\ \delta_{\Theta(\omega)}(I_\omega^\varepsilon)=1.
\]
So, for all $n$ sufficiently large, $(g_1\circ\dots\circ g_n)^{-1}(I_\omega^\varepsilon)\cap(0,x)\neq \emptyset$, and similarly for $(y,1)$.

Hence, by monotonicity of $g_i$, we have
$$
\limsup_{n\to\infty} |g_{1}\circ\cdots\circ g_{n} ([x, y ])|\le \limsup_{n\to\infty} \bigl| g_1\circ\cdots\circ g_n(x_n) - g_1\circ\cdots\circ g_n(y_n)\bigl|
\le 2\varepsilon
$$
for $\mu^{\otimes \mathbb N}$--a.e. $(g_{1}, g_{2},\ldots)$.
Since $\varepsilon>0$ was arbitrarily small, we obtain
$$
\lim_{n\to\infty} |g_{1}\circ\cdots\circ g_{n} ([x, y ])|=0\qquad\text{for $\mu^{\otimes \mathbb N}$--a.e. $(g_{1}, g_{2},\ldots)$}.
$$

Fix now $\nu_1, \nu_2\in\mathcal P$ and let  $\varphi: [0, 1]\to\mathbb R$ be a Lipschitz function with Lipschitz constant $L_\varphi>0$. Put $\bar{\nu}:=(\nu_1+\nu_2)/2$. Let $\varepsilon>0$. The C\'esaro averages $(\bar{\nu} + P\bar{\nu}+\dots+P^{n-1}\bar{\nu})/n$ converge weakly to a $\mu$--invariant measure $\nu_*$. By assumption $\nu_*$ is atomless. Thus, for some $N'\in\N$, $P^{N'}\bar{\nu}((0,1))>1-\varepsilon/(8L_\varphi)$. Consequently, $P^{N'}\nu_i\bigl((0,1)\bigr)>1-\varepsilon/(4L_\varphi)$ for $i=1,2$. Put $\nu'_i:=P^{N'}\nu_i$.

Then for $n\geq N$ we have
$$
\begin{aligned}
&\left|\langle P^{n+N'}\nu_1, \varphi\rangle-\langle P^{n+N'}\nu_2, \varphi\rangle\right|
= \left|\langle P^{n}\nu'_1, \varphi\rangle-\langle P^{n}\nu'_2, \varphi\rangle\right|\\&\le L_\varphi \int_{\Gamma^n}\int_ {[0, 1]^2} |g_1\circ \cdots\circ g_n(x)-g_1\circ \cdots\circ g_n(y)|\mu^{\otimes n}(\d g_1,\ldots, \d g_n)\nu'_1(\d x)\nu'_2(\d y)\\
&\le L_\varphi \int_{\Gamma^{\mathbb N}}\int_{(0, 1)^2} |g_1\circ \cdots\circ g_n([x, y])|\mu^{\otimes \mathbb N}(\d g_1, \d g_2, \ldots)\nu_1(\d x)\nu_2(\d y) \\
& \qquad\  + \ 2L_\varphi(\nu'_1\otimes\nu'_2)\bigl([0,1]^2\setminus (0,1)^2\bigr) 
\end{aligned}
$$
On the other hand, we have
\[
(\nu'_1\otimes\nu'_2)\bigl([0,1]^2\setminus (0,1)^2\bigr)\ \le \ \nu'_1(\{0,1\}) + \nu'_2(\{0,1\})\  <\ 2\varepsilon/(4L_\varphi)=\varepsilon/(2L_\varphi).
\]
By Lebesgue's Dominated Convergence Theorem we obtain that
$$
\begin{aligned}
&\limsup_{n\to\infty} \left|\langle P^n\nu_1, \varphi\rangle-
\langle P^n\nu_2, \varphi\rangle\right|\\
& \le L_\varphi \int_{\Gamma^n}\int_{(0, 1)^2} \lim_{n\to\infty} |g_1\circ \cdots\circ g_n([x, y])|\mu^{\otimes \mathbb N}(\d g_1, \d g_2, \ldots)\nu_1(\d x)\nu_2(\d y)\\
&+\varepsilon/(2L_\varphi)=\varepsilon/(2L_\varphi),
\end{aligned}
$$
and since $\varepsilon>0$ was arbitrary condition (\ref{e1_4.10.25}) is satisfied, and the proof of stability is completed.
\end{proof}

\begin{rem}
    A particular consequence of Proposition \ref{prop:monotone case asympt stab} is that there is only one $\mu$--invariant measure $\nu$. Moreover for any continuous function $\varphi\in C([0, 1])$ and $\tilde\nu\in\mathcal P$ we will have
    $$
        \lim_{n\to\infty} \langle P^n\tilde\nu, \varphi\rangle=
      \langle \nu, \varphi\rangle.
    $$
\end{rem}

\begin{rem}
    Let $\Gamma$ consist of strictly monotonic functions. For any Borel probability measure $\mu$ on $\Gamma$ the stochastic dynamical system $(\Gamma,\mu)$ is then $\mu$--injective. If $\Gamma$ also satisfies our key condition~\ref{A}, then any $\mu$--invariant measure is atomless, according to Theorem~\ref{thm: atomless}. Thus, a crucial assumption of Proposition \ref{prop:monotone case asympt stab} is satisfied. Assume additionally that $\supp\mu=\Gamma$. 
     
    If there exists now a $\mu$--invariant measure $\nu$ that is $\mu$--proximal and has $\{0,1\}\subset \supp\nu$, then Proposition \ref{prop:monotone case asympt stab} implies that the Markov operator of $(\Gamma,\mu)$ is asymptotically stable. In particular, $\nu$ is the only $\mu$--invariant measure of $(\Gamma,\mu)$.
    
    In view of Theorem \ref{thm_21_05_24}, if also $(\Gamma,\mu)$ is below the diagonal close to 0 or above the diagonal  close to 1 (hence satisfies \ref{A+}), then any $\mu$--invariant Borel probability measure $\nu$ is $\mu$--proximal. We cannot conclude that $\{0,1\}\subset\supp\nu$ though, which would allow us to conclude by means of Proposition \ref{prop:monotone case asympt stab} that the Markov operator corresponding to $(\Gamma,\mu)$ is asymptotically stable.
\end{rem}
\vskip 2mm

In the previous proof, we used crucially the fact that $\Gamma$ consists of strictly monotonic functions.
Now we are going to prove asymptotic stability of the Markov operator of some $\mu$--injective stochastic dynamical systems without the necessity for $\Gamma$ to consist of strictly monotonic functions. Here a different approach is required.

The next lemma says that the iterates of the Markov operator starting at an arbitrary probability are concentrated around any point from the support of a $\mu$--invariant measure. Its proof is in the spirit of the lower bound technique derived by Lasota and Yorke (see \cite{L_Y, Sz}). 

\begin{lem}\label{L1_27.05.24}
Let $\Gamma$ satisfy \ref{A+}, and let $\mu$ be a Borel probability measure on $\Gamma$ with $\supp\mu=\Gamma$.
Assume that $(\Gamma,\mu)$ has a unique $\mu$--invariant Borel probability measure $\nu$. Then for every point $x_0\in\supp\nu$ and open set $I\ni x_0$ there exist $\alpha>0$ and $n\in\mathbb N$ such that
$$
P^{n}\tilde\nu(I)\ge\alpha
$$
for every probability measure $\tilde\nu\in\mathcal P$, for $P$ defined by (\ref{e1_31.08.24}).
\end{lem}

\begin{proof} {The idea of this proof is quite simple. One could easily see that the action of maps on the support of an invariant measure is minimal (the neighbourhood of every point in the support is reachable from another point and the transition probability is bounded from below by some positive constant). Thus one needs to evaluate the probability of entrance in the neighbourhood of some point of the support. Since all this takes place on a compact set this positive probability may be bounded from below by a positive constant. Putting together these two effects would conclude the proof.}
\\
Let $(\Gamma,\mu)$ be given, and let $\nu$ be its unique $\mu$--invariant Borel probability measure. Fix $x\in\supp\nu$ and let $I$ be an open neighbourhood of $x$. Since 
$$
\frac{\delta_0+P\delta_0+\cdots+ P^{n-1}\delta_0}{n}
$$
converges weakly to $\nu$, by the Portmanteau Theorem there exists $m_0\in\mathbb N$ such that
 $$
 P^{m_0}\delta_0 (I)>\nu(I)/2:=\gamma
 $$
 and, consequently, we have for all $x$ sufficiently close to $0$
 \[
P^{m_0}\delta_x (I)>\gamma,
 \]
 by the weak continuity of the operator $P^{m_0}$. 

 Since the system is above/below the diagonal close to zero/one, we may choose $b>0$ such that $\mu(\{g\in \Gamma: g([0, b))\subset [0, b) \})>0$.
 Further, by \ref{A} (see also Lemma \ref{r1_02.10.25}) for every $y\in [0, 1]$ we may find $m(y)\in\mathbb N$ and a sequence ${\bf g}_y=(g_1,\ldots, g_{m(y)})\in\Gamma^{m(y)}$ such that $g_{1}\circ\cdots\circ g_{m(y)}(y)\in [0, b)$. Hence, by continuity, there exist open neighbourhood  $U_y$ of $y$ and $V_{{\bf g}_y}\subset \Gamma^{m(y)}$ of ${\bf g}_y$ such that for all $z\in U_y$ and $(h_1,\ldots, h_{m(y)})\in V_{{\bf g}_y}$ we have
 $h_{1}\circ\cdots\circ h_{m(y)} (z)\in [0, b)$, by Lemma \ref{r1_02.10.25}. Obviously,
 $\mu^{\otimes m(y)}(V_{{\bf g}_y})>0$. By compactness of $[0, 1]$ we choose $x_1,\ldots, x_k\in [0, 1]$ such that $[0, 1]\subset\bigcup_{i=1}^k U_{x_i}$, and let $m=\max_{1\le i\le k} m(x_i)$. Let $\hat\Gamma:=\{g\in\Gamma: g([0, b))\subset [0, b)\}$. Set $V_y:=V_{{\bf g}_y}\times \hat\Gamma^{m-m(y)}$ and note that $\mu^{\otimes m}(V_y)>0$. We have also $h_1\circ\cdots\circ h_m(U_y)\subset [0, b)$ for every $(h_1,\ldots, h_m)\in V_y$. Put
 $$
 \kappa:=\min_{1\le i\le k} \mu^{\otimes m}(V_{x_i})
 $$
 and observe that 
 $$
  P^m\delta_y([0, b))\ge\idotsint_{V_{x_i}} {\bf 1}_{[0, b)}(g_1\circ\cdots\circ g_m(y))\mu^{\otimes m}(\d (g_1,\ldots, g_m))\ge\kappa,
  $$
 where $x_i$ is such that $y\in U_{x_i}$. Hence
 for every $\tilde\nu\in\mathcal P$ we have
 $$
 P^m\tilde\nu([0, b))=\int_{\supp\tilde\nu} P^m\delta_y([0, b))\tilde\nu(\d y)\ge\kappa.
 $$
  Finally, from the Markov property we obtain
  $$
  P^{m+m_0}\tilde\nu(I)\ge\int_{[0, b)}P^{m_0}\delta_y(I)P^m\tilde\nu(\d y)\ge \gamma\kappa:=\alpha
  $$
  for every probability measure $\tilde\nu\in\mathcal P$. This completes the proof.
  \end{proof}

We are in a position to prove the stability of the considered system.

\begin{thm}[Theorem~\ref{thm: stab intro}]\label{thm} Let $\Gamma$ satisfy condition~\ref{A+}, and let $\mu$ be a Borel probability measure on $\Gamma$ with $\supp\mu=\Gamma$. Then its corresponding Markov operator is asymptotically stable.
 \end{thm}
 
 \begin{proof} { In the proof we make use of the lower bound technique developed in \cite{L_Y, Sz}. In fact, one could notice that iterating an arbitrary operator of a measure an $\alpha$--part of it would be brought to some small interval, by Lemma \ref{L1_27.05.24}. On the other hand, since that interval is arbitrary, it might be chosen in such a way that, by Proposition~3.3 in \cite{BOS}, with high probability  the composition of maps would decrease its length. This would lead to a conclusion that after iterating two different measures one might subtract in both of them $\alpha$-part in such a way that these parts would be ever close to each other. Iterating this scheme one would obtain the conclusion.  
 }

 Since there exists a $\mu$--invariant probability measure $\nu$, to prove stability it is enough to show that for every $\varepsilon>0$ and any two Borel probability measures $\tilde\nu_1, \tilde\nu_2$ on $[0,1]$, we have
 \begin{equation}\label{e1_27.05.24}
 \limsup_{n\to\infty} d_{W}(P^n\tilde\nu_1, P^n\tilde\nu_2)<\varepsilon.
 \end{equation}
 
Now fix $\varepsilon>0$. According to Theorem \ref{T_1_03.05.25} we may  choose $\theta>0$ such that the set of all $(g_1,\ldots, g_n, \ldots)\in \Gamma^{\mathbb{N}}$
for which 
\begin{equation}\label{e1_9.10.23}
|g_n\circ \ldots \circ g_1((x_0-\theta/2, x_0+\theta/2))|\leq \exp(-n\cdot \widetilde{h})
\end{equation}
has the measure $\mu^{\mathbb{N}}$ greater than $1-\varepsilon/4$. Denote this set by $\tilde\Gamma^{\mathbb{N}}$ and let $\tilde\Gamma^{n}$ be its projection on $\Gamma^n$, for $n\in\mathbb N$.

At the beginning, we prove condition~(\ref{e1_27.05.24}) for any two measures $\nu_1$ and $\nu_2$ supported on the set $I$. To do this take a $1$--Lipschitz function $\varphi$. Taking if necessary the function $\varphi-\varphi(0)$, we may assume additionally that $\varphi(0)=0$. Therefore, for any $m\in\mathbb N$ we have
  \begin{equation}\label{e2_27.05.24}
  \begin{aligned}
  &\left |\langle P^m\nu_1, \varphi\rangle
  -\langle P^m\nu_2, \varphi\rangle \right|\\
  &\le 
  \idotsint_{\Gamma^m\times [0, 1]^2} |\varphi(g_m\circ\cdots\circ g_1(x))-
  \varphi(g_m\circ\cdots\circ g_1(y))|\mu^{\otimes m}(\d (g_1,\ldots, g_m))\nu_1(\d x)\nu_2(\d y)\\
  &\le \idotsint_{\tilde\Gamma^m\times [0, 1]^2} |g_m\circ\cdots\circ g_1(x))-
  g_m\circ\cdots\circ g_1(y)|\mu^{\otimes m}(\d (g_1,\ldots, g_m))\nu_1(\d x)\nu_2(\d y)\\
  &+\idotsint_{(\Gamma^m\setminus\tilde\Gamma^m)\times [0, 1]^2} 2\,\mu^{\otimes m}(\d (g_1,\ldots, g_m))\nu_1(\d x)\nu_2(\d y)
  \le \exp(-m\cdot \widetilde{h})+\tfrac{2\varepsilon}{4},
  \end{aligned}
 \end{equation}
  where the last inequality is obtained by (\ref{e1_9.10.23}) and the fact that $\mu^{\otimes m}(\Gamma^m\setminus\tilde\Gamma^m)<\varepsilon/4$.

  By Lemma \ref{L1_27.05.24} there exists $n_0\ge 1$ and $\alpha>0$ such that  
  $$
  P^{n_0}\tilde\nu(I)\ge\alpha
  $$
  for every $\tilde\nu\in\mathcal P$. Consequently for any measure $\tilde\nu\in\mathcal P$ we have two probability measures $\tilde\nu^+$ and $\tilde\nu^-$ such that $\tilde\nu^+$ is supported on $I$ and
  \begin{equation}\label{e11_09.10.23}
  P^{n_0}\tilde\nu=\alpha\tilde\nu^+ +(1-\alpha)\tilde\nu^-.
\end{equation}
  In fact, we define
  $$
  \tilde\nu^+(\cdot)=\frac{P^{n_0}\tilde\nu(\cdot\cap I)}{P^{n_0}\tilde\nu( I)}
  $$
  and
  $$
  \tilde\nu^-(\cdot)=\tfrac{1}{1-\alpha}\left(P^{n_0}\tilde\nu(\cdot)-\alpha\tilde\nu^+(\cdot)\right).
  $$
  Now equation (\ref{e11_09.10.23}) leads
  for $n\ge n_0$ to:
  $$
  \begin{aligned}
  d_W(P^n\tilde\nu_1, P^n\tilde\nu_2)&\le \alpha d_W(P^{n-n_0}\tilde\nu^+_1, P^{n-n_0}\tilde\nu^+_2)\\
  &+(1-\alpha) d_W(P^{n-n_0}\tilde\nu^-_1, P^{n-n_0}\tilde\nu^-_2).
  \end{aligned}
  $$
  By (\ref{e2_27.05.24}) for any $\tilde\nu_1, \tilde\nu_2\in\mathcal P$ we have
  \begin{equation}\label{e3_09.10.23}
  \limsup_{n\to\infty} d_W(P^n\tilde\nu_1, P^n\tilde\nu_2)\le (1-\alpha)\limsup_{n\to\infty} d_W(P^n\tilde\nu^-_1, P^n\tilde\nu^-_2)+\alpha \varepsilon.
  \end{equation}
  Set $\varDelta:=\sup_{\tilde\nu_1, \tilde\nu_2\in\mathcal P} \limsup_{n\to\infty} d_W(P^n\tilde\nu_1, P^n\tilde\nu_2)\le 2$ and note that (\ref{e3_09.10.23}) gives
  $$
  \varDelta\le (1-\alpha)\varDelta+\alpha\varepsilon
  $$
  and consequently $\varDelta\le\varepsilon$. Hence
  $$
  \limsup_{n\to\infty} d_W(P^n\nu_1, P^n\nu_2)\le\varepsilon\quad\text{for every $\nu_1, \nu_2\in\mathcal P$},
  $$
  and since $\varepsilon>0$ was arbitrary, we finally obtain
  $$
  \lim_{n\to\infty} d_W(P^n\tilde\nu_1, P^n\tilde\nu_2)=0
  $$
  for arbitrary measures $\tilde\nu_1, \tilde\nu_2\in\mathcal P$.
 This completes the proof.
 \end{proof}

\section{Synchronization}\label{Synchronization}

We start with the lemma proving that the number of iterations required for attaining a neighbourhood of a point in the support of a $\mu$--invariant measure $\nu$ may be chosen independently of the initial points. In consequence, we may prove that the probability of entering into this neighbourhood is bounded from below.

\begin{lem}\label{l1_14.06.25}
 Let $\Gamma$ satisfy \ref{A+}, and let $\mu$ be a Borel probability measure on $\Gamma$ with $\supp\mu=\Gamma$. Let $\nu$ be a $\mu$--invariant measure, and let $I$ be an open set in $[0, 1]$ such that $I\cap\supp\nu\neq\emptyset$.
Then there exist $K\in\mathbb N$ and $\alpha>0$ such that for every $x, y\in [0, 1]$ one may find a subset $\Gamma^K (x, y)\subset \Gamma^K$ such that $\mu^{\otimes K}(\Gamma^K (x, y))\ge\alpha$ and
\[
g_{1}\circ\cdots\circ g_{K}(x)\in I\quad\text{and}\quad g_{1}\circ\cdots \circ g_{K}(y)\in I
\]
for $(g_{1},\ldots, g_{K})\in \Gamma^K(x, y)$.
\end{lem}

\begin{proof} Let $\nu$ be a $\mu$--invariant measure. Fix an open set $I$ such that $I\cap\supp\nu\neq\emptyset$. By Theorem \ref{thm} the operator $P$ corresponding to $(\Gamma, \mu)$ is asymptotically stable and 
\[
\liminf_{n\to\infty} P^n\delta_0 (I)\ge\nu_*(I)>0,
\]
by the Portmanteau Theorem.
Hence there exists $m\ge 1$ such that
\[
P^{m}\delta_0(I)=\int_{\Gamma^{m}} {\bf 1}_I( g_1\circ\cdots\circ g_m (0))\mu^{\otimes m}(\d (g_1,\ldots, g_m))>0.
\]
Therefore, there exist $(h_1, \ldots, h_m)$ such that $h_1\circ\cdots\circ h_m (0)\in I$ and, by continuity, an open neighbourhood $\Gamma^m(h_1, \ldots, h_m)\subset\Gamma^m$ of $(h_1,\ldots, h_m)$ and $\varepsilon>0$ such that $g_1\circ\cdots\circ g_m([0, \varepsilon))\subset I$ for any $(g_1,\ldots, g_1)\in \Gamma^m(h_1,\ldots, h_m)$. There is no loss of generality in assuming that $\varepsilon<\varepsilon_0$, where
 $\varepsilon_0$ is such that 
\[
\mu(\{g\in\Gamma_0: g(z)<z\,  \text{ for }\, z\in (0, \varepsilon_0)\})>0.
\] 
Since the support of $\mu^{\otimes m}$ is equal to $\Gamma^m$, we obtain 
\[
\alpha_0:=\mu^{\otimes m}(\Gamma^m(h_1,\ldots, h_m))>0.
\]
 
By Lemma \ref{l1_04.10.25} and continuity of all the functions from $\Gamma$ we find for every $u, v\in [0, 1]$ a positive integer $m$, a sequence $(g_1, \ldots, g_m)\in\Gamma_0^m$  and open neighbourhoods $O_u$ of $u$ and $O'_v$ of $v$ such that
\[
g_1\circ\cdots\circ g_m (O_u)\in [0, \varepsilon)\qquad\text{and   }\quad g_1\circ\cdots\circ g_m (O'_v)\in [0, \varepsilon).
\]
Again, by continuity of maps from $\Gamma$, we may assume that the above 
condition holds for some open neighbourhood of $(g_1,\ldots g_m)$, that is, on a $\mu^{\otimes m}$--positive set. 

Consequently, by compactness we may find a finite covering of $[0, 1]\times [0, 1]$ by sets $O_{u_i}\times O'_{v_j}$ such that there exist $m_{i, j}\ge 1$ and subsets $\Gamma^{m_{i, j}} (i, j)\subset \Gamma^{m_{i, j}}$ with $\mu^{\otimes m_{i, j}}(\Gamma^{m_{i, j}} (i, j))>0$ and $f_1\circ\cdots\circ f_{m_{i, j}} (O_{u_i})\subset [0, \varepsilon)$ and $f_1\circ\cdots\circ f_{m_{i, j}} (O'_{v_j})\subset [0, \varepsilon)$ for $(f_1,\ldots, f_{m_{i, j}})\in \Gamma^{m_{i, j}} (i, j)$.

Although the values of $m_{i, j}$ for different pairs $(i, j)$ 
may be different, by adjusting, if necessary,  the iterates of maps 
from the set $\{g\in\Gamma_0: g(z)<z\,  \text{ for }\, z\in (0, \varepsilon_0)\}$, we may assume that all $m_{i, j}$ are equal to some $\kappa$. Indeed, if $f_0\in \{g\in\Gamma_0: g(z)<z\,  \text{ for }\, z\in (0, \varepsilon_0)\}$, then $f_0^l\circ f_1\circ\cdots\circ f_{m_{i, j}} (O_{u_i})\subset [0, \varepsilon)$ and $f_0^l\circ f_1\circ\cdots\circ f_{m_{i, j}} (O'_{v_j})\subset [0, \varepsilon)$ for every $l\ge 1$.
Since the set $\{g\in\Gamma_0: g(z)<z\,  \text{ for }\, z\in (0, \varepsilon_0)\}$ is a $\mu$--positive set, we may also assume that $\mu^{\otimes \kappa}(\Gamma^{\kappa} (i, j))>0$.

Let $K=m+\kappa$. Setting 
$\alpha_1:=\min_{i, j}\mu^{\otimes \kappa}(\Gamma^{\kappa} (i, j))$ and $\alpha:=\alpha_0\cdot\alpha_1$, 
we obtain the desired assertion:
 for every points $x, y\in [0, 1]$  and the set $\Gamma^K(x, y)=\Gamma^m(h_1\ldots h_m)\times \Gamma^{\kappa} (i, j)$, where $i, j$ are such that $x\in O_{u_i}$ and $y\in O'_{v_j}$, we have
\[
g_{1}\circ\cdots\circ g_{K}(x)\in I\quad\text{and}\quad g_{1}\circ\cdots \circ g_{1}(y)\in I
\]
for $(g_{1},\ldots g_{K})\in \Gamma^K(x, y)$. Finally, note that $\mu^{\otimes K}(\Gamma^K(x, y))\ge\alpha_0\cdot\alpha_1=\alpha.$ 
The proof is complete.
\end{proof}

The expected value with respect to $\mu^{\otimes\mathbb N}$ will be denoted by $\mathbb E$. By $\mathbb E$ we shall also denote the expected value with respect to $\mu^{\otimes m}$ if this is not misleading.

Set
\[
S_n(x, y; {\bf g}):=\sum_{i=1}^n | g_i\circ\cdots\circ g_1 (x)-g_i\circ\cdots\circ g_1 (y)|
\]
for $x, y\in [0, 1]$, ${\bf g}=(g_1, \ldots, g_n)\in\Gamma^n $ and $n\in\mathbb N$, and analogously
\[
S_{\infty}(x, y; {\bf g}):=\sum_{n=1}^{\infty} | g_n\circ\cdots\circ g_1 (x)-g_n\circ\cdots\circ g_1 (y)|
\]
for $x, y\in [0, 1]$, ${\bf g}=(g_1, g_2, \ldots)\in\Gamma^{\mathbb N}$.
\vskip3mm

The following proposition is a profound modification of 
Lemma 6 in \cite{G_S}. 

 \begin{prop}\label{p1_28.05.24}
 Let $\Gamma$ satisfy \ref{A+}, and let $\mu$ be a Borel probability measure on $\Gamma$ with $\supp\mu=\Gamma$. 
Then there exists a constant $C$ such that for any $x, y\in [0, 1]$ we have
$$
\mathbb E S_{\infty} (x, y; \cdot)\le C.
$$
\end{prop}

\begin{proof} Let us denote
\[
a_n:=\sup_{x, y\in [0, 1]}\mathbb E S_n(x, y; \cdot)\qquad\text{for $n\in\mathbb N$.}
\]
Obviously $(a_n)_{n\ge 1}$ is a nondecreasing sequence and $a_n\le n$ for $n\in\mathbb N$. By Theorem \ref{T_1_03.05.25} there exist an open interval $I\subset [0, 1]$, $I\cap\supp\nu\neq\emptyset$, a constant $q\in (0, 1)$ and a set $\Gamma_*^{\mathbb N}\subset \Gamma^{\mathbb N}$ with $\beta:=\mu^{\otimes \mathbb N} (\Gamma_*^{\mathbb N})>0$ such that $|g_n\circ\cdots\circ g_1 (I)|\le q^n$ for $(g_1, g_2, \ldots) \in\Gamma_*^{\mathbb N}$. By 
Lemma \ref{l1_14.06.25}, in turn, we may find $K\in\mathbb N$ and $\alpha>0$ such that for every $x, y\in [0, 1]$ there exists $\Gamma^K(x, y)\subset\Gamma^K$ with $\mu^{\otimes K}(\Gamma^K(x, y))>\alpha$ and \[
g_{K}\circ\cdots\circ g_{1}(x)\in I\quad\text{and}\quad g_{K}\circ\cdots \circ g_{1}(y)\in I
\]
for $(g_{1},\ldots g_{K})\in \Gamma^K(x, y)$.

For $n > K$ and arbitrary $x, y\in [0, 1]$ we have
\begin{align}
\begin{split}\label{eq:partsum}
&\mathbb E S_n(x, y; \cdot)=
\int_{\Gamma^n} S_n(x, y; {\bf g})\mu^{\otimes n}(\d {\bf g}) \\
& \le \int_{\Gamma^{n-K}}\int_{\Gamma^K (x, y)} (K+ S_{n-K} (\tilde {\bf g} (x), \tilde {\bf g} (y); {\bf g}))\mu^{\otimes K}(\d \tilde {\bf g})\mu^{\otimes n-K}(\d {\bf g})\\
&+ \int_{\Gamma^{n-K}}\int_{(\Gamma^K\setminus \Gamma^K(x, y))} (K+ S_{n-K} (\tilde {\bf g} (x), \tilde {\bf g} (y); {\bf g}))\mu^{\otimes K}(\d \tilde {\bf g})\mu^{\otimes n-K}(\d {\bf g})\\
&\le 
\int_{\Gamma^{n-K}}\int_{\Gamma^K(x, y)} (K+ S_{n-K} (\tilde {\bf g} (x), \tilde {\bf g} (y); {\bf g}))\mu^{\otimes K}(\d \tilde {\bf g})\mu^{\otimes n-K}(\d {\bf g})\\
&+(K+a_n) \mu^{\otimes n} ((\Gamma^K\setminus \Gamma^K (x, y)\times\Gamma^{n-K}),\\
\end{split}
\end{align}
the last inequality by the fact that $a_{n-K}\le a_n$.

In the sequel we shall use the following notation. By $\Gamma_*^m$ we denote the projection of $\Gamma_*^{\mathbb N}$ on product $\Gamma^m$, $m\in\mathbb{N}$, i.e.,
\begin{displaymath}
\Gamma_*^m := \left\{ (g_1,\ldots, g_m)\in \Gamma^m : (g_1,\ldots, g_m)\times \Gamma^\mathbb{N}\cap \Gamma_*^{\mathbb N} \neq \emptyset \right\}.
\end{displaymath}
For arbitrary $z,w\in I$ and $r \le n$ we have
\begin{align*}
 \mathbb E S_r(z, w; \cdot) &= \int_{\Gamma^r} S_r(z, w; {\bf g})  \mu^{\otimes r} (\d {\bf g}) 
 = \int_{\Gamma_*^r} S_r(z, w; {\bf g})  \mu^{\otimes r}(\d {\bf g})\\
&+
 \int_{\Gamma^r \setminus \Gamma_*^r} S_r(z, w; {\bf g}) \mu^{\otimes r}(\d {\bf g})
=: \text{I+II}.
\end{align*}
Firstly we evaluate integral II. For $k < r$ we set
\begin{displaymath}
\Gamma_k^r := \left\{ (g_1,\ldots,g_r)\in \Gamma^r : (g_1,\ldots, g_{k-1})\in\Gamma_*^{k-1}, (g_1,\ldots, g_{k})\notin\Gamma_*^k \right\}.
\end{displaymath}
The set $\Gamma_k^r$ contains all sequences which have the first $k-1$ elements from some sequence of $\Gamma_*$ but on $k$-th position they are different. Note that $\Gamma^r\setminus \Gamma_*^r = \bigcup\limits_{k=1}^r \Gamma_k^r$ and $\Gamma_i^r\cap \Gamma_j^r = \emptyset$ if $i\neq j$. We easily see that $\mu^{\otimes r}(\Gamma^r\setminus \Gamma_*^r)\le 
\mu^{\otimes \mathbb N}(\Gamma^{\mathbb N}\setminus\Gamma_*)=1-  \mu^{\otimes \mathbb N}(\Gamma_*)\le 1-\beta$.
Further, let ${\Gamma_k^r}_{|k}$ be a projection of $\Gamma_k^r$ on $k$ first elements, i.e.,
\begin{displaymath}
{\Gamma_k^r}_{|k} := \left\{ (g_1,\ldots, g_k)\in \Gamma^k : \exists (g_1,\ldots, g_{k}, g_{k+1},\ldots, g_r)\in \Gamma_k^{r}\right\}.
\end{displaymath}
Then we see that $\Gamma_k^r = {\Gamma_k^r}_{|k} \times \Gamma^{r-k}$ and
\[
\text{II} =\int_{\bigcup\limits_{k=1}^r \Gamma_k^r} S_r(z, w; {\bf g})\mu^{\otimes r}(\d {\bf g})= \sum\limits_{k=1}^r \int_{\Gamma_k^r} S_r(z, w; {\bf g})\mu^{\otimes r}(\d {\bf g}) .
\]
Let us consider one term of the sum above. For $k\in \{1,\ldots r\}$ we have from Fubini's theorem
\[
\begin{split}\label{eq:Gamma_k^r}
&  \int_{\Gamma_k^r} S_r(z, w; {\bf g}) \mu^{\otimes r}(\d {\bf g})\\
& = \int_{{\Gamma_k^r}_{|k}}\left(\int_{\Gamma^{r-k}}S_r(z, w; {\bf g}) \mu(\d g_{k+1})\cdots \mu(\d g_r)\right)\mu^{\otimes k}(\d {\bf g}).
\end{split}
\]
Now to evaluate the inner integrals we set $\hat{\bf g}=(g_1,\ldots, g_k)$ and $\check{\bf g}=(g_{k+1},\ldots, g_r)$, and let $\hat z: =g_k\circ\dots\circ g_1(z)$ and $\hat w: =g_k\circ\dots\circ g_1(w)$. Then we have
\begin{align*}
& \int_{\Gamma^{r-k}} S_r(z, w; {\bf g}) \mu^{\otimes r-k}(\d \check{\bf g})\\
& = \idotsint\limits_{\Gamma^{r-k}} \left( S_k(z, w; \hat{\bf g})+S_{r-k-1}(\hat z, \hat w; \check{\bf g})\right) \mu(\d g_{k+1})\cdots \mu(\d g_r)\\
& = S_k(z, w; \hat{\bf g})+\idotsint\limits_{\Gamma^{r-k}} S_{r-k-1}(\hat z, \hat w; \check{\bf g}) \mu(\d g_{k+1})\cdots \mu(\d g_r)
\end{align*}
and consequently
\begin{align*}
&  \int_{\Gamma_k^r} 
S_r(z, w; {\bf g}) \mu^{\otimes r}(\d {\bf g})  \\
&=  \int_{{\Gamma_k^r}_{|k}}\left(S_k(z, w; \hat{\bf g})+\int_{\Gamma^{r-k}} S_{r-k}(\hat z, \hat w; \check{\bf g}) \mu^{\otimes (r-k) }(\d \check{\bf g})\right)
\mu^{\otimes k}(\d \hat{\bf g})  \\
&\le\int_{{\Gamma_k^r}_{|k}} ( S_k(z, w; \hat{\bf g}) +a_{r-k})\mu^{\otimes k}(\d \hat{\bf g})
\le\int_{{\Gamma_k^r}_{|k}} ( S_k(z, w; \hat{\bf g}) +a_{r})\mu^{\otimes k}(\d \hat{\bf g})\\
& \le \left ( (q + q^2 + \ldots + q^{k}) + a_{r}\right) \mu^{\otimes k} ({\Gamma_k^r}_{|k})\le \left(\tfrac{q}{1-q}+a_{r}\right) \mu^{\otimes k}({\Gamma_k^r}_{|k}).
\end{align*}
Therefore,
\begin{align*}
\text{II} & \le 
 \sum\limits_{k=1}^r \int_{\Gamma_k^r} S_r(z, w; {\bf g}) \mu^{\otimes r}(\d {\bf g})  \le 
\sum_{k=1}^r\left(\tfrac{q}{1-q}+a_{r}\right) \mu^{\otimes k}({\Gamma_k^r}_{|k})
\\
& \le \tfrac{q}{1-q} + a_r \sum\limits_{k=1}^r \mu^{\otimes k}({\Gamma_k^r}_{|k})
 \le \tfrac{q}{1-q} + a_r (1-\beta).
\end{align*}

On the other hand, integral I can be easily evaluated by contractivity. Indeed, we have
\begin{displaymath}
\text{I} =\int_{\Gamma_*^r}S_r(z, w; {\bf g}) \mu^{\otimes r}(\d {\bf g}) \le \left(q + q^2 + \ldots + q^{r}\right) \le \tfrac{q}{1-q}.
\end{displaymath}
Hence for $z,w \in I$
\begin{equation}\label{e1_18.06.25}
\mathbb E S_r(z, w; \cdot) \le \text{I} + \text{II} \le \tfrac{2q}{1-q} + a_r(1-\beta).
\end{equation}

Now from (\ref{eq:partsum}) and (\ref{e1_18.06.25}) for $n>K$, by Fubini's theorem,  we obtain
\[
\begin{split}
\mathbb E S_n(x, y; \cdot)
&\le 
\int_{\Gamma^{n-K}}\int_{\Gamma^K( x, y)} \left(K+ S_{n-K} (\tilde {\bf g} (x), \tilde {\bf g} (y); {\bf g})\right)\mu^{\otimes K}(\d \tilde {\bf g})\mu^{\otimes n-K}(\d {\bf g})\\
&+(K+a_n) \mu^{\otimes n} ((\Gamma^K\setminus \Gamma^K (x, y))\times\Gamma^{n-K})\\
&=\int_{\Gamma^K (x, y)} (K+\int_{\Gamma^{n-K}} S_{n-K} (\tilde {\bf g} (x), \tilde {\bf g} (y); {\bf g})\mu^{\otimes n-K}(\d {\bf g})) \mu^{\otimes K}(\d \tilde {\bf g})\\
&+(K+a_n) \mu^{\otimes K} (\Gamma^K\setminus \Gamma^K (x, y))\\
&\le \mu^{\otimes K}(\Gamma^K (x, y)) (K+\tfrac{2q}{1-q} +a_n (1-\beta))+(K+a_n) (1-\mu^{\otimes K}(\Gamma^K (x, y))\\
&\le  a_n (1-\beta\mu^{\otimes K}(\Gamma^K (x, y)))+ 2K+\tfrac{2q}{1-q}\le (1-\alpha\beta) a_n + 2K+\tfrac{2q}{1-q}
\end{split}
\]
and since $x, y\in [0, 1]$ was arbitrary, taking supremum on the left hand side over $x, y$, yields
\[
a_n\le  (1-\alpha\beta) a_n + 2K+\tfrac{2q}{1-q}\quad \text{ for $n>K$.} 
\]
Since $a_n\le K$ for $n\le K$ the above inequality is satisfied for all $n\in\mathbb N$. Thus
\[
a_n\le \frac{2K+\tfrac{2q}{1-q}}{\alpha\beta}:=C\qquad\text{for $n\ge 1$}
\]
and by the Lebesgue monotone convergence theorem
\[
\mathbb E S_{\infty} (x, y; \cdot)\le \lim_{n\to\infty} a_n\le C\quad\text{ for $x, y\in [0, 1]$. }
\]
This completes the proof.
\end{proof}

\begin{rem} Under the assumptions of Proposition \ref{p1_28.05.24} the system $(\Gamma, \mu)$ satisfies the synchronization condition, i.e., for every $x, y\in [0, 1]$
\[
| g_n\circ\cdots\circ g_1 (x)- g_n\circ\cdots\circ g_1 (y)|\to 0
\]
for $\mu^{\otimes \mathbb N}$--a.e. $(g_1, g_2,\ldots)\in\Gamma^{\mathbb N}$. Indeed, we have
\[
\int_{\Gamma^{\mathbb N}}\left(\sum_{n=1}^{\infty}  | g_n\circ\cdots\circ g_1 (x)- g_n\circ\cdots\circ g_1 (y)|\right)\mu^{\otimes \mathbb N} (\d (g_1, g_2, \ldots))\le C
\]
and hence $\sum_{n=1}^{\infty}  | g_n\circ\cdots\circ g_1 (x)- g_n\circ\cdots\circ g_1 (y)|<\infty$ for $\mu^{\otimes \mathbb N}$--a.e. $(g_1, g_2,\ldots)\in\Gamma^{\mathbb N}$. This implies the synchronization of the system.
\end{rem}

With the same proof as in Proposition \ref{p1_28.05.24} we may also obtain:

\begin{rem}\label{r1_21.06.25}
Let $\varphi\colon [0, 1]\to\mathbb R$ be an arbitrary H\"older function. Set
\[
S_n^{\varphi} (x, y; {\bf g}):=\sum_{i=1}^n |\varphi(g_i\circ\cdots\circ g_1 (x)) - \varphi(g_i\circ\cdots\circ g_1 (y))|
\]
for $x, y\in [0, 1]$, ${\bf g}=(g_1, \ldots, g_n)\in\Gamma^n $ and $n\ge 1$, and analogously
\[
S_{\infty}^{\varphi} (x, y; {\bf g}):=\sum_{n=1}^{\infty} | \varphi(g_n\circ\cdots\circ g_1 (x))-\varphi(g_n\circ\cdots\circ g_1 (y))|
\]
for $x, y\in [0, 1]$, ${\bf g}=(g_1, g_2, \ldots)\in\Gamma^{\mathbb N}$. If the assumptions of Proposition \ref{p1_28.05.24} hold, then 
there exists a constant $C$ such that for any $x, y\in [0, 1]$ we have
$$
\mathbb E S_{\infty}^{\varphi} (x, y; \cdot)\le C.
$$
\end{rem}

By use of Proposition \ref{p1_28.05.24} we may give a straightforward proof of the next lemma.

\begin{lem}\label{l1_7.08.24}
Let $\Gamma$ satisfy \ref{A+}, and let $\mu$ be a Borel probability measure on $\Gamma$ with $\supp\mu=\Gamma$.
Assume that $(\Gamma, \mu)$ is $\mu$--injective, and let $\nu$ be its unique $\mu$--invariant measure. If
$\varphi: [0, 1]\to\mathbb R$ is a centered Lipschitz function, then the function $\Psi(x):= \sum_{n=0}^{\infty} U^n \varphi(x)$ for $x\in [0, 1]$ is a well defined bounded function, and consequently $\Psi\in L^2(\nu)$. Moreover $\varphi=\Psi -U\Psi$ ($\varphi$ is a $L^2(\nu)$--coboundary).
\end{lem}

\begin{proof} Fix a Lipschitz function $\varphi: [0, 1]\to\mathbb R$, and let $L_{\varphi}$ denotes its Lipschitz constant. Let  $C>0$ be the constant derived in Proposition \ref{p1_28.05.24}.
We have
 \[
 \begin{aligned}
 \sum_{i=0}^{n} |U^i \varphi(x)|&=
\sum_{i=0}^{n} |U^i \varphi(x)-\langle P^i\nu, \varphi\rangle|\\
&\le \sum_{i=0}^{n} \int_{[0, 1]}|U^i \varphi(x)-U^i \varphi(z)|\nu(\d z)\\
&\le L_{\varphi} \int_{[0, 1]}  \mathbb E S_n(x, z; \cdot)\nu(\d z)\le  L_{\varphi} C.
\end{aligned}
\]
Thus $\Psi(x):= \sum_{n=0}^{\infty} U^n \varphi(x)$ for $x\in [0, 1]$ is a well defined bounded function, and consequently $\Psi\in L^2(\nu)$. It is immediate to check that
$\Psi-U\Psi=\varphi$. 
The proof is complete.\end{proof}

\section{Central Limit Theorem}\label{sec: clt}

In the last section we are concerned with the central limit theorem for the corresponding Markov chain. Having the stochastic dynamical system $(\Gamma, \mu)$ and the  Markov operator $P$ and its pre-dual $U$ we define  the  Markov chain $(\xi_n)_{n\ge 0}$ in the following way: the law of $(\xi_n)_{n\ge 0}$ with initial distribution $\nu_0$ is the probability measure $\mathbb{P}_{\nu_0}$ on $\left( [0, 1]^\mathbb{N}, \mathcal{B}([0, 1])^{\otimes\mathbb{N}} \right)$ such that
\begin{displaymath}
\mathbb{P}_{\nu_0}(\xi_{n+1}\in A | \xi_n = x) = P\delta_x(A) \quad \text{and} \quad \mathbb{P}_{\nu_0}(\xi_0\in A) = \nu_0(A),
\end{displaymath}
where $x\in [0, 1]$ and $A\in\mathcal{B}([0, 1])$. The existence of $\mathbb{P}_{\nu_0}$ follows from the Kolmogorov extension theorem. When an initial distribution $\nu_0$ is equal to a $\mu$--invariant measure, the Markov chain is called stationary.

We have
\[
\text{if $\xi_0=x$\,\,\, then}\quad \xi_n(\omega)=g_n\circ\cdots\circ g_1 (x)\quad\text{for  $\omega=(g_1, g_2, \ldots)\in\Gamma^{\mathbb N}$}
\]

The expectation with respect to $\mathbb P_{\nu_0}$ is denoted by $\mathbb E_{\nu_0}$. For $\nu_0=\delta_x$, the Dirac measure at $x\in [0, 1]$, we write just $\mathbb P_x$ and $\mathbb E_x$.
Obviously $\mathbb P_{\nu_0}(\cdot)=\int_{[0, 1]}\mathbb P_x(\cdot)\nu_0(\d x)$ and $\mathbb E_{\nu_0}(\cdot)=\int_{[0, 1]}\mathbb E_x(\cdot)\nu_0(\d x)$. 
We will prove that the stochastic process $(\varphi(\xi_n))_{n\ge 0}$, where $(\xi_n)_{n\ge 0}$ denotes the corresponding stationary Markov chain with initial distribution $\nu$ ($\nu$ is a $\mu$-invariant measure) and $\varphi$ is a {\it centered Lipschitz function} on $[0, 1]$, i.e.,
$\int_{[0, 1]}\varphi(x)\nu(\d x)=0$, satisfies the {central limit theorem}, i.e., 
\[
\sigma^2:=\lim_{n\to\infty} \mathbb{E}_{\nu}\bigg( \frac{\varphi(\xi_1)+\cdots+\varphi(\xi_n)}{\sqrt{n}} \bigg)^2
\]
exists and
\[
\frac{\varphi(\xi_1)+\cdots+\varphi(\xi_n)}{\sqrt{n}}\Rightarrow \mathcal{N}(0,\sigma^2)\qquad\text{as $n\to\infty$,}
\]
where $\Rightarrow$  denotes convergence in distribution.
Moreover, we will be able to show that  $(\varphi(\xi^x_n))_{n\ge 1}$, where $(\xi^x_n)_{n\ge 1}$ is the corresponding Markov chain starting from an arbitrary point $x\in [0, 1]$, satisfies the {central limit theorem from every point}, i.e.,
$$
\frac{\varphi(\xi_1^x)+\cdots+\varphi(\xi_n^x)}{\sqrt{n}}\Rightarrow \mathcal{N}(0,\sigma^2)\qquad\text{as $n\to\infty$}.
$$
Therefore, if $\sigma^2>0$ we have
\begin{equation}\label{c1_03.02.21}
\begin{aligned}
\lim_{n\to\infty} \mathbb \mu^{\otimes\mathbb N}&\left( \left\{(g_1, g_2\ldots)\in\Gamma^{\mathbb N} : \frac{\varphi(g_1(x))+\cdots+\varphi(g_n\circ\cdots\circ g_1(x))}{\sqrt{n}}<a\right\}\right)\\
&=\frac{1}{\sqrt{2\pi\sigma^2}}\int_{-\infty}^a e^{-\frac{y^2}{2\sigma^2}}\d y
\end{aligned}
\end{equation}
for all $a\in\mathbb R$ and $x\in [0, 1]$. If $\sigma=0$, the sequence
\[
\frac{\varphi(g_1(x))+\cdots+\varphi(g_n\circ\cdots\circ g_1(x))}{\sqrt{n}}
\]
converges in distribution to $0$.
\vskip3mm
For $\omega=(g_1, g_2, \ldots)\in\Gamma^{\mathbb N}$ set
$$
{\mathfrak S}_n(\varphi, \omega, x):=\frac{\varphi(g_1(x))+\cdots+\varphi(g_n\circ\cdots\circ g_1(x))}{\sqrt{n}}.
$$

Finally, we are in a position to formulate and prove the central limit theorem. Having proved Proposition \ref{p1_28.05.24} its proof will be straightforward.

\begin{thm}[Theorem \ref{thm: CLT intro}]\label{thm: clt}
Let $\Gamma$ satisfy \ref{A+}, and let $\mu$ be a Borel probability measure on $\Gamma$ with $\supp\mu=\Gamma$.
Moreover, assume that $(\Gamma, \mu)$ is $\mu$--injective. Let $(\xi_n)_{n\ge 0}$ be the corresponding stationary Markov chain with initial distribution $\nu$.
Then for any centered Lipschitz function $\varphi: [0, 1]\to\mathbb R$ the random process $(\varphi(\xi_n))_{n\ge 1}$ satisfies the central limit theorem.
Moreover, the same is true for the process $(\varphi(\xi^x_n))_{n\ge 1}$, where $(\xi^x_n)_{n\ge 1}$ is the corresponding Markov chain starting from an arbitrary point $x\in [0, 1]$.
\end{thm}

\begin{proof} Fix a centered Lipschitz function $\varphi: [0, 1]\to\mathbb R$. Since $\varphi$ is an $L^2(\nu)$--coboundary of $U$, by Lemma \ref{l1_7.08.24},  $\sigma^2$ exists and is finite, by \cite{Gordin_Lifsic}. Moreover, condition~(\ref{c1_03.02.21}) holds for $\nu\almost \; x_0\in [0, 1]$, by Gordin and Lif\v{s}ic (see Section IV.8 in \cite{B_I}, see also \cite{D_L}). In particular, for $\nu\almost\; x_0\in [0, 1]$ we have
\begin{equation}\label{e1_28.05.24}
\lim_{n\to\infty}\int_{\Gamma^{\mathbb N}}\exp( it {\mathfrak S}_n(\varphi, \omega, x_0)) \mu^{\otimes\mathbb N} (\d\omega)=\exp\left(-\tfrac{1}{2}t^2\sigma^2\right)\quad\text{for $t\in\mathbb R$}.
\end{equation}
Fix such $x_0$ such that (\ref{e1_28.05.24}) holds and observe that for any $x\in [0, 1]$ we have
\begin{equation}\label{e1_21.06.25}
\begin{aligned}
&\left|\int_{\Gamma^{\mathbb N}}\exp( it {\mathfrak S}_n(\varphi, \omega, x)) \mu^{\otimes\mathbb N} (\d\omega)-\int_{\Gamma^{\mathbb N}}\exp( it {\mathfrak S}_n(\varphi, \omega, x_0)) \mu^{\otimes\mathbb N} (\d\omega)\right|\\
&\le \tfrac{L_{\varphi}}{\sqrt n}  \sum_{i=1}^{n}\mathbb E S_i(x, x_0; \cdot) \le \tfrac{L_{\varphi} C}{\sqrt n},
\end{aligned}
\end{equation}
where $C>0$ is the constant derived in Proposition \ref{p1_28.05.24} and $L_{\varphi}$ denotes the Lipschitz constant of the function $\varphi$. Thus convergence (\ref{e1_28.05.24}) holds for every $x\in [0, 1]$ and the proof is completed.
\end{proof}

\begin{remark} 
The above central limit theorem holds for any centered H\"older function $\varphi: [0, 1]\to\mathbb R$. Indeed, in evaluating inequality (\ref{e1_21.06.25}) we use then Remark \ref{r1_21.06.25} instead of Proposition \ref{p1_28.05.24}. In the proof that $\varphi$ is an $L^2$--coboundary (a version of Lemma \ref{l1_7.08.24}) we also make use of Remark \ref{r1_21.06.25}.
\end{remark}

\begin{remark} There are plenty of examples which fulfill the conditions of the foregoing theorems. To illustrate that the choice of the distribution $\mu$ is crucial, consider Example~\ref{ex: 3maps}. Condition~\ref{A+} is satisfied. If  $p\geq \frac{2}{5}$, then the system is $\mu$--injective, thus there exists a unique $\mu$--invariant Borel probability measure (Corollary~\ref{c1_21_05_24}). This measure is $\mu$--proximal (Theorem~\ref{thm_21_05_24}) and atomless (Theorem~\ref{thm: atomless}). Moreover, the Markov operator corresponding to that system is asymptotically stable (Theorem~\ref{thm}) and the central limit theorem described in Theorem~\ref{thm: clt} holds.\\
  But if on the other hand $0<p<\frac{1}{4}$, then there are at least two different ergodic $\mu$--invariant Borel probability measures on $[0,1]$, by Proposition~4.4 in \cite{BOS}.
\end{remark}

\end{document}